\documentclass[11pt]{amsart}
\usepackage{amsmath}
\usepackage{latexsym, bm}
\usepackage{indentfirst}
\usepackage{geometry}
\usepackage{array}
\usepackage{lipsum}
\usepackage{setspace}
\usepackage{multirow}
\usepackage{tabu}
\usepackage{amsfonts}
\usepackage{amsthm}
\usepackage{fancyhdr}
\usepackage{tikz}
\usepackage{amscd}
\usepackage{amssymb}
\usepackage{enumerate}

\usepackage{amssymb,amsfonts}
\usepackage[all,arc]{xy}
\usepackage{enumerate}
\usepackage{mathrsfs}
\usepackage{geometry}
\usepackage{amsmath}
\usepackage{latexsym, bm}
\usepackage{indentfirst}
\usepackage{CJK}
\usepackage{amsthm}
\usepackage{hyperref}
\usepackage{cite}

\usepackage{amsmath}
\usepackage{latexsym, bm}
\usepackage{indentfirst}
\usepackage{geometry}
\usepackage{array}
\usepackage{lipsum}
\usepackage{setspace}
\usepackage{multirow}
\usepackage{tabu}
\usepackage{amsfonts}
\usepackage{amsthm}
\usepackage{fancyhdr}
\usepackage{tikz}
\usepackage{amscd}
\usepackage{amssymb}

\usepackage{tikz-cd}
\usepackage{extarrows}

\usepackage{amssymb,amsfonts}
\usepackage[all,arc]{xy}
\usepackage{enumerate}
\usepackage{mathrsfs}
\usepackage{geometry}
\usepackage{amsmath}
\usepackage{latexsym, bm}
\usepackage{indentfirst}
\usepackage{CJK}
\usepackage{amsthm}
\usepackage{hyperref}
\usepackage{cite}

\def \p{\partial}

\newtheorem{thm}{Theorem}[section]
\newtheorem{cor}[thm]{Corollary}
\newtheorem{lem}[thm]{Lemma}
\newtheorem{prop}[thm]{Proposition}

\newtheorem{defn}[thm]{Definition}
\newtheorem{remk}[thm]{Remark}

\newtheorem{que}[thm]{Question}

\DeclareMathOperator{\ep}{\epsilon}

\DeclareMathOperator{\C}{\mathbb C}
\DeclareMathOperator{\R}{\mathbb R}
\DeclareMathOperator{\rk}{rk}
\DeclareMathOperator{\tr}{tr}

\begin{document}
\title{\small Kodaira dimensions of almost complex manifolds II}
\author{Haojie Chen }
\address{Department of Mathematics\\ Zhejiang Normal University\\ Jinhua Zhejiang, 321004, China}
\email{chj@zjnu.edu.cn}
\author{Weiyi Zhang}
\address{Mathematics Institute\\  University of Warwick\\ Coventry, CV4 7AL, England}
\email{weiyi.zhang@warwick.ac.uk}

\begin{abstract}
This is the second of a series of papers where we study the plurigenera, the Kodaira dimension and the Iitaka dimension on compact almost complex manifolds. By using the pseudoholomorphic pluricanonical map, we define the second version of Kodaira dimension as well as Iitaka dimension on compact almost complex manifolds. We show the almost complex structures with the top Kodaira dimension are integrable. For compact almost complex 4-manifolds with Kodaira dimension one, we obtain elliptic fibration like structural description. Some vanishing theorems in complex geometry are generalized to the almost complex setting.

For  tamed symplectic $4$-manifolds, we show that the almost Kodaira dimension is bounded above by the symplectic Kodaira dimension, by using a probabilistic combinatorics style argument. The appendix also contains a few results including answering a question of the second author that there is a unique subvariety in exceptional curve classes for irrational symplectic $4$-manifolds, as well as extending Evans-Smith's constraint on symplectic embeddings of certain rational homology balls by removing the assumption on the intersection form. 

\end{abstract}

\date{}
\maketitle

\tableofcontents

\section{Introduction}
The Iitaka dimension for a holomorphic line bundle $L$ over a compact complex manifold is a numerical invariant to measure the size of the space of holomorphic sections. It could be  defined as the growth rate of the dimension of the space $H^0(X, L^{\otimes d})$, or equivalently the maximal image dimension of the rational map to projective space determined by powers of $L$, or $1$ less than the dimension of the section ring of $L$. When $L$ is the canonical bundle $\mathcal K_X$, the corresponding Iitaka dimension is called the Kodaira dimension of $X$.

In the first paper of this series \cite{CZ}, we have generalized the notions of  the Iitaka dimension and Kodaira dimension to compact almost complex manifolds using the growth rate of the pseudoholomorphic sections of the corresponding complex line bundle endowed with a pseudoholomorphic structure, or equivalently, a bundle almost complex structure.

More precisely, let $E$ be a complex vector bundle  over an almost complex manifold $(X,J)$. A bundle almost complex structure $\mathcal J$ on $E$ is a special almost complex structure on the total space of the $E$, see \cite{DeT, CZ}.  There is a bijection between bundle almost complex structures and the pseudoholomorphic structures on $E$. The $(E, \mathcal J)$-genus is defined as $P_{E, \mathcal J}:=\dim H^0(X, (E, \mathcal J)),$ where $H^0(X, (E, \mathcal J))$ denotes the space of $(J,\mathcal J)$ pseudoholomorphic sections. As shown by Hodge theory in \cite{CZ}, $P_{E, \mathcal J}$ is of finite dimension. The $m^{th}$ plurigenus of $(X,J)$ is defined to be $P_m(X, J)=\dim H^0(X, \mathcal K_X^{\otimes m})$.

\begin{defn}\label{Iv1}

Let $L$ be a complex line bundle with bundle almost complex structure $\mathcal J$ over $(X, J)$. The Iitaka dimension $\kappa^J(X, (L, \mathcal J))$ is defined as
$$\kappa^J(X, (L, \mathcal J))=\begin{cases}\begin{array}{cl} -\infty, &\ \text{if} \ P_{L^{\otimes m},\mathcal J}=0\  \text{for any} \ m\geq 0\\
\limsup_{m\rightarrow \infty} \dfrac{\log P_{L^{\otimes m},\mathcal J}}{\log m}, &\  \text{otherwise.}
\end{array}\end{cases}$$

The Kodaira dimension $\kappa^J(X)$ is defined by choosing $L=\mathcal K_X$ and $\mathcal J$ to be the bundle almost complex structure induced by $\bar{\p}$.
\end{defn}

In this paper, we will introduce the second version of the Iitaka dimension and Kodaira dimension on compact almost complex manifolds by generalizing the equivalent definition in the complex setting which is the maximal image dimension of the rational map to projective space by the holomormphic sections of $L^{\otimes m}$. In the complex setting, important structural results and speculations were obtained using this interpretation.

Assume that $L$ is a complex line bundle over a compact almost complex manifold $(X,J)$ with a bundle almost complex structure $\mathcal J$. As in complex setting, the map $\Phi_{L, \mathcal J}: X\setminus B\rightarrow \mathbb CP^N$  is defined as $\Phi_{L, \mathcal J}(x)=[s_0(x): \cdots :s_{N}(x)]$, where $s_i$ constitute a basis of the linear space $H^0(X, (L, \mathcal J))$, and $B\subset X$ is the set of the base locus, {\it i.e.} the set of points $x\in X$ such that $s_i(x)=0$ for all $0\le i\le N$. When $L$ is $\mathcal K_X^{\otimes m}$ with the standard bundle almost complex structure, $\Phi_{L, \mathcal J}$ is denoted by $\Phi_m$ and called the {\it pluricanonical map}.

We first show that $\Phi_{L, \mathcal J}$ is pseudoholomorphic between $X\setminus B$ and the standard projective space. Following the study of pseudoholomorphic maps from an almost complex manifolds to a complex manifold, we prove that the image of $\Phi_{L, \mathcal J}(X\setminus B)$ is an open analytic subvariety in $\mathbb CP^N$. As a result, we define
\begin{defn}\label{Iv2}
The Iitaka dimension $\kappa_J(X, (L, \mathcal J))$ of a complex line bundle $L$ with bundle almost complex structure $\mathcal J$ over $(X, J)$ is defined as
$$\kappa_J(X, (L, \mathcal J))=\begin{cases}\begin{array}{cl} -\infty, &\ \text{if} \ P_{L^{\otimes m}}=0\  \text{for any} \ m\geq 0\\
\max\dim_{\C} \Phi_{L^{\otimes m}} (X\setminus B), &\  \text{otherwise.}\end{array}
\end{cases}$$

The Kodaira dimension $\kappa_J(X)$ is defined by choosing $L=\mathcal K_X$ and $\mathcal J$ to be the bundle almost complex structure induced by $J$.
\end{defn}
We have shown that the plurigenera $P_m(X,J)$ are birational invariants on compact almost complex 4-manifolds in \cite{CZ}. It follows that $\kappa_J(X)$ is also a birational invariant in dimension $4$ (Theorem \ref{Kodbir'}).

We then use the Kodaira dimension $\kappa_J(X)$ to study the properties of $(X,J)$. By definition, $\kappa_J(X, (L, \mathcal J))\le \dim_{\C}X$. When the top value $\dim_{\C}X$ is achieved, we prove that $J$ is integrable. The following is Theorem \ref{Koddimmax}.

\begin{thm}
Let $(X, J)$ be a connected almost complex manifold and $L$ be a complex line bundle with a bundle almost complex structure $\mathcal J$.

\begin{enumerate}
\item If $\kappa_J(X, (L, \mathcal J))=\dim_{\C} X$, then $J$ is integrable.
\item If $P_{L, \mathcal J}\ge 2$, then $\kappa_J(X, (L, \mathcal J))\ge 1$.

\end{enumerate}
\end{thm}

Although a non-integrable almost complex structure cannot have $\kappa_J$ being equal to its complex dimension, the examples of compact $2n$-dimensional nonintegrable almost complex manifolds in Section 6.3 of our first paper \cite{CZ}  also realize all the values of $\{-\infty, 0, 1, \cdots, n-1\}$ for $\kappa_J(X)$ (as well as $\kappa^J(X)$)  and any $n\geq 2$. The examples are based on a $4$-dimensional example with Kodaira dimension $1$, which is a deformation of the product complex structure on $T^2\times \Sigma_{g\ge 2}$.

For almost complex manifolds with line bundles whose almost complex Iitaka dimensions take value between $1$ and $n-1$, we expect that there are important structural result of $X$. The following is such a result (Theorem \ref{ellsur}).

\begin{thm}\label{ellintro}
 Let $(X, J)$ be a connected almost complex 4-manifold. If the pluricanonical map is base-point-free and $\kappa_J(X)=1$, $X$ admits a pseudoholomorphic fibration with each fiber an elliptic curve and finitely many singular fibers.
\end{thm}

The examples of almost complex $4$-manifolds with $\kappa^J=\kappa_J=1$ contain all the smooth pseudoholomorphic elliptic fibrations over a Riemann surface of genus greater than one with tamed almost complex structure.

In dimension $4$, there is a version of symplectic Kodaira dimension $\kappa^s$ \cite{L}, which measures the positivity of the symplectic canonical class (see Definition \ref{symKod}). This Kodaira dimension is a diffeomorphism invariant of any symplectic $4$-manifold {\it loc. cit.}, and it coincides with the complex Kodaira dimension if the symplectic $4$-manifold admits a complex structure \cite{DZ}. The invariant provides a rough classification scheme for symplectic $4$-manifolds \cite{L}. It is studied extensively in $4$-manifolds theory, {\it e.g.} its behaviour under various symplectic surgeries like fiber sum, Luttinger surgery, and rational blowdown (see the references in \cite{Li15}). It is also used to show some smooth structures are exotic.

It is natural to ask the connections between the almost complex Kodaira dimensions and the symplectic Kodaira dimension in dimension $4$. We are able to bound the Kodaira dimension $\kappa^J(X)$ from above by $\kappa^s$ at least when the almost complex structure $J$ is tamed.  Recall that an almost complex structure $J$ is said to be tamed if there is a symplectic form $\omega$  such that the bilinear form $\omega(\cdot, J\cdot)$ is positive definite. To establish this result, we introduce an intermediate Iitaka dimension $\kappa_W^J(X, e)$ for tamed almost complex structure $J$ and $e\in H^2(M, \mathbb Z)$. The Weil divisor version of Kodaira dimension is defined as $\kappa_W^J(X):=\kappa_W^J(X, K_X)$, where $K_X$ is the first Chern class of $\mathcal K_X$. Roughly speaking, this version of Iitaka dimension measures the growth rate of the numbers of all pseudoholomorphic curves in the multiples of the class $e$. Among various estimates of this Iitaka dimensions in Section \ref{Weil}, we have the following comparison results for Kodaira dimensions.

\begin{thm}\label{JWs}
Let $(X, J)$ be a tamed almost complex $4$-manifold. Then $$\kappa^J(X)\le \kappa_W^J(X)\le \kappa^s(X).$$
\end{thm}

As a corollary (Corollary \ref{W<J}), for any tamed almost complex structure $J$ on a complex surface $(X, J_X)$, $\kappa^J(X)\le \kappa^{J_X}(X).$

Theorem \ref{JWs}  is obtained through estimations of pseudoholomorphic curves in multiples of the canonical class using a very rough and ``overdetermined" parametrization of the moduli space pseudoholomorphic subvarieties in class $mK_X$. These points given by the parametrization fix a unique element in the moduli space and, at the same time, the number of the points as $m$ grows is bounded above by a polynomial whose degree matches up with symplectic Kodaira dimension.
This is very similar to the style of estimates in probabilistic combinatorics.

We next study how the positivity of line bundles is related to the almost complex Iitaka dimension. The relation between positivity/negativity of certain curvatures and vanishing of groups of holomorphic sections has been an active topic in complex geometry, at least dating back to Kodaira vanishing theorem.  A negative line bundle on a compact complex manifold, {\it i.e.} a holomorphic line bundle whose first Chern class is a negative class, always has negative Iitaka dimension. To obtain its natural extension to almost Hermitian setting, we apply the intersection theory of almost complex manifolds developed in \cite{Z2}. With the built up there and Bochner type arguments on almost Hermitian manifolds, we prove that a complex line bundle over compact almost complex manifolds does not admit any pseudoholomorphic sections if it satisfies some negative conditions. Using similar arguments, we also obtain the vanishing of plurigenera for compact almost complex manifolds admitting Hermitian metrics with some positive curvature conditions, which generalizes the known results in the complex setting (in the K\"ahler case by Yau \cite{Yau} and in Hermitian case by Yang \cite{Y}). The results are summarized in the following (a combination of Theorem \ref{c1<0van} and Theorem \ref{metric}).
\begin{thm}
Let $(X, J)$ be a compact almost complex $2n$-manifold and $L$ be a complex line bundle over $X$.
\begin{enumerate}
\item If there is a $J$-compatible symplectic form $\omega$ such that $[\omega]^{n-1}\cdot c_1(L)<0$, then for any bundle almost complex structure $\mathcal J$ on $L$, $P_{L, \mathcal J}=0$ and $\kappa^J(X, (L, \mathcal J))=\kappa_J(X, (L, \mathcal J))=-\infty$.
\item If $X$ admits a Hermitian metric with positive Chern scalar curvature everywhere or a Gauduchon metric with positive total Chern scalar curvature, then $\kappa^J(X)=\kappa_J(X)=-\infty$.
\end{enumerate}
\end{thm}

The last section of the paper contains some speculations and open problems generated from the study of our almost complex Iitaka dimensions. It is divided into 5 subsections on different aspects. It seems that our paper is just a start of a vast area with a rich source of interesting problems for almost complex manifolds. It is related to many other fields such as Hodge theory, birational geometry, complex geometry, and topology of manifolds. We expect Section \ref{pro} would provide the reader with a glance of these interesting problems.

Finally, there is an appendix on the uniqueness of exceptional curves for irrational symplectic $4$-manifolds. Its main result gives an affirmative answer to a question of the second author \cite{p=h}.

\begin{thm}\label{A1}
Let $M$ be a symplectic $4$-manifold which is not diffeomorphic to $\mathbb CP^2\#k\overline{\mathbb CP^2}, k\ge 1$. Then for any tamed $J$, there is a unique $J$-holomorphic subvariety in any exceptional curve class $E$, whose irreducible components are smooth rational curves of negative self-intersection. Moreover, this $J$-holomorphic subvariety is an exceptional curve of the first kind when $M$ is also not diffeomorphic to $S^2\times \Sigma_g\#k\overline{\mathbb CP^2}$.
\end{thm}
On the other hand, as noticed in \cite{p=h, s2mod} as a somewhat surprising result, there are many integrable complex structures on rational surfaces where this statement does not hold. The link of the appendix with this paper is through Proposition \ref{1pcE}, which says that there is a unique $J$-holomorphic subvariety in any positive combination of exceptional curve classes for any tamed almost complex structure $J$ on a symplectic $4$-manifold which is not diffeomorphic to a rational or ruled surface. It is used in the proof of Theorem \ref{JWs} when the symplectic Kodaira dimension is zero.

As a by-product of the proof of Theorem \ref{A1}, we are able to improve the result of \cite{ES} on the symplectic version of the upper bound of Wahl singularities by removing their assumption on the dimension of self-dual harmonic forms.

The first author is partially supported by NSFC grant (11901530) and Zhejiang Provincial Natural Science Foundation (LY19A010017). He would like to thank Professors Kefeng Liu, Xiaokui Yang and Fangyang Zheng for their helpful conversations. The second author would like to thank Giancarlo Urz\'ua for his interest and discussion on Theorem \ref{ES1}.

Without otherwise mentioning, our objects are connected closed oriented manifolds.

\section{Pluricanonical maps}

In this section, we will give another definition of Iitaka dimension for almost complex manifolds. Regarding to Kodaira dimension, it is the maximum of the image of pluricanonical map. For this sake, we first derive some general properties of the image of a pseudoholomorphic map onto a complex manifold in Section \ref{imageph}. These are used to establish structural results for pluricanonical maps.

\subsection{Pseudoholomorphic subvarieties}\label{imageph}
In \cite{Z2}, we give a definition of pseudoholomorphic subvarieties for an arbitrary even dimension.
A $J$-holomorphic subvariety $\Theta$ of an almost complex manifold $(M, J)$ is a finite set of pairs $\{(V_i, m_i), 1\le i\le m\}$, where each $V_i$ is an irreducible $J$-holomorphic subvariety and each $m_i$ is a positive integer. Here an irreducible $J$-holomorphic subvariety is the image of a somewhere immersed pseudoholomorphic map $\phi: X\rightarrow M$ from a compact connected smooth almost complex manifold $X$. When each $V_i$ has complex dimension one, we call $\Theta$ a complex $1$-subvariety.

When $(M, J)$ is a complex manifold, the source $(X, J_X)$ could still be an almost complex manifold. Hence, it seems there are more $J$-holomorphic subvarieties than complex analytic subvarieties at a first glance. However, we observe that when $(M, J)$ is a complex manifold,  a $J$-holomorphic subvariety is exactly a complex analytic subvariety.

\begin{prop}\label{comsubvar}
If there is a somewhere immersed pseudoholomorphic map $\phi: X\rightarrow M$ from a compact connected smooth almost complex manifold $(X, J_X)$ to a complex manifold $(M, J)$, then $J_X$ is integrable. Therefore, a finite set of pairs $\{(V_i, m_i), 1\le i\le m\}$ on $(M, J)$ is a $J$-holomorphic subvariety if and only if it is a complex analytic subvariety.
\end{prop}
\begin{proof}
We will use a similar method as for Theorem 5.5 of \cite{Z2}. However, since $X$ and $M$ are not equidimensional in general, we will need to do extra projections. We can cover $M$ by complex coordinate patches. Hence, we can assume our target is an open subset $U\subset \C^n$, with coordinates $(z_1, \cdots, z_n)$. Since $\phi$ is somewhere immersed, we have $\dim_{\C} X=m\le n$. Let $\{\sigma_1, \cdots, \sigma_m\}$ be a choice of $m$ elements in $\{1, \cdots, n\}$. Denote the corresponding image of $U$ by $U_{\sigma_1, \cdots, \sigma_m}$ or simply $U_{\C^m}$ if there is no confusion. Also denote the projection of $\phi|_{X_U}$ to $\C^m_{z_{\sigma_1}, \cdots, z_{\sigma_m}}$ components by $pr_{\sigma_1\cdots\sigma_m}\phi$.

For each such choice of $\{\sigma_1, \cdots, \sigma_m\}$, we look at the  projection from $\C^n$ to $m$ coordinates $(z_{\sigma_1}, \cdots, z_{\sigma_m})$. This gives us a complex vector bundle over $X_U\times U_{\C^m}$ whose fiber over $(x, pr_{\sigma_1\cdots\sigma_m}\phi(x))$ is the complex vector space of all complex linear maps $L_{\sigma_1\cdots\sigma_m}: T_xX\rightarrow \C^m$. Here $X_U$ is the set $\phi^{-1}(U)\subset X$.  By taking fiberwise complex determinant, we have a complex line bundle $\mathcal L_{\sigma_1\cdots\sigma_m}$ over $X_U\times U_{\C^m}$, whose fibers are $\det L_{\sigma_1\cdots\sigma_m}: \Lambda_{\C}^mT_xX\rightarrow \Lambda_{\C}^m\C^m\cong \C$. Its total space has a standard almost complex structure, see \cite{Z2}.

 Then the pseudoholomorphic map $\phi$ induces pseudoholomorphic maps $\phi_{\mathcal L_{\sigma_1\cdots\sigma_m}}(x)=(x, \phi(x), \det (d(pr_{\sigma_1\cdots\sigma_m}\phi)_x)_{\mathbb C})$ from $X_U$ to $\mathcal L_{\sigma_1\cdots\sigma_m}$.  Hence, the singularity subset of $\phi$ (in $X_U$) is the intersection of $\phi^{-1}_{\mathcal L_{\sigma_1\cdots\sigma_m}}(X_U\times U_{\C^m} \times \{0\})$ for all possible subsets $\{\sigma_1, \cdots, \sigma_m\}$. Applying Theorem 3.8 of \cite{Z2}, we know each $\phi^{-1}_{\mathcal L_{\sigma_1\cdots\sigma_m}}(X_U\times U_{\C^m} \times \{0\})$ is a subset with finite $2(m-1)$-dimensional Hausdorff measure if it is not the whole $X_U$. Since $\phi$ is somewhere immersed and $X$ is connected, there must be $\{\sigma_1, \cdots, \sigma_m\}$ such that $\phi^{-1}_{\mathcal L_{\sigma_1\cdots\sigma_m}}(X_U\times U_{\C^m} \times \{0\})$ is not the whole $X_U$ (but possibly it is empty). It implies the singular subset $\mathcal S_{\phi}$ has finite $2(m-1)$-dimensional Hausdorff measure. In particular, the closure of $X\setminus \mathcal S_{\phi}$ is $X$.

On $X\setminus \mathcal S_{\phi}$, the map $\phi$ is an immersion. In particular, it is local diffeomorphic to its image and thus the complex coordinates of $(M, J)$ could be pulled back by $\phi$. It implies the Nijenhuis tensor $N_{J_X}|_{X\setminus \mathcal S_{\phi}}=0$. Since  the closure of $X\setminus \mathcal S_{\phi}$ is $X$, we know $N_{J_X}=0$. That is $J_X$ is integrable.

The second statement follows from a theorem of Remmert \cite{Rem} which says that the image of a proper holomorphic map between two complex manifolds is a complex analytic subvariety of the target. Hence, it implies that if $\{(V_i, m_i), 1\le i\le m\}$ on $(M, J)$ is a $J$-holomorphic subvariety, then it is a complex analytic subvariety. The reverse direction follows from Hironaka's resolution of singularities. This completes the proof of the lemma.
\end{proof}
This proof also shows that the image of the singularity subset $\phi(\mathcal S_{\phi})$ is a complex analytic subvariety up to Question 3.9 of \cite{Z2}, and thus a $J$-holomorphic subvariety of $M$. When $\phi$ is a  general proper pseudoholomorphic map, we can still have a similar conclusion as Proposition \ref{comsubvar}, although we only have a less precise description of the image of the singularity subset.

\begin{prop}\label{imageprop}
Let $\phi: X\rightarrow M$ be a proper pseudoholomorphic map from a closed almost complex manifold $(X, J_X)$ to a complex manifold $(M, J)$. Then the image of $\phi$ is a complex analytic subvariety.
\end{prop}
\begin{proof}
We look at the differential of $d\phi: TX\rightarrow TM$ and view it as a linear transformation at each point of $X$. Since $\phi$ is a pseudoholomorphic map, the rank of this linear transformation at each point is an even integer no greater than $\min\{2\dim_{\C}X, 2\dim_{\C}M\}$. Define the rank of $\phi$ by $\rk(\phi)=\frac{1}{2}\sup_{x\in X}\rk(d\phi_x)$. We call any point $x\in X$ with $\rk(\phi)=\frac{1}{2}\rk(d\phi_x)$ a regular point and any other point a critical point.

We apply a result of Sard \cite{Sar}: If $f: N_1\rightarrow N_2$ is a $C^k$ map for $k\ge \max\{n-m+1, 1\}$ and if $A_r\subset N_1$ is the set of points $x\in N_1$ such that $df_x$ has rank no greater than $r$, then the $(r+\ep)$-dimensional Hausdorff measure of $\phi(A_r)$ is zero, for any $\ep>0$.

In our setting, we choose $r=2\rk(\phi)-2$. Thus $A_r$ is the set of critical points. In particular, the above version of Sard's theorem implies that the Hausdorff measure $\mathcal H^{2\rk(\phi)-1}(\phi(A_r))=0$.

On the other hand, the set of  regular points is an open dense subset in $X$. The openness is apparent and the denseness means $\overline{X\setminus A_r}=X$. To show the denseness, we look at the $\rk(\phi)^{th}$ wedge product of both the complex vector bundles $TX$ and $TM$, denoted by $\Lambda_{\C}^{\rk(\phi)}TX$ and $\Lambda_{\C}^{\rk(\phi)}TM$. Then the $1$-jet bundle $\mathcal E=J^1(X, M)$ induces a complex vector bundle $\mathcal E^{\rk(\phi)}$ (over $X\times M$) whose fibers are induced complex linear maps from $\Lambda_{\C}^{\rk(\phi)}T_xX$ to $\Lambda_{\C}^{\rk(\phi)}T_{\phi(x)}M$. Its total space inherits a standard almost complex structure from the standard almost complex structure $\mathcal J$ of $J^1(X, M)$. The pseudoholomorphic map $\phi$ induces a pseudoholomorphic section $(x, \phi(x), (d\phi_x)_{\C})$ of $J^1(X, M)$ and in turn a pseudoholomorphic section of $\mathcal E^{\rk(\phi)}$. The zero locus of this section is just $A_r$. By the extension of unique continuation for pseudoholomorphic submanifolds, {\it i.e.} Proposition 2.3 of \cite{Z2}, we know $A_r$ cannot contain an open subset of $X$ since it is the intersection of two almost complex submanifolds, {\it i.e.} the image of the zero section and the induced section $(x, \phi(x), \Lambda^{\rk(\phi)}_{\C}(d\phi_x)_{\C})$. Since $X$ is a manifold, we have $\overline{X\setminus A_r}=X$.

Applying the rank theorem (see {\it e.g.} \cite{Lee}), we know the image of an open neighborhood of $x\in X\setminus A_r$ is diffeomorphic to $\R^{2\rk(\phi)}=\R^{r+2}$. Since $\phi$ is pseudoholomorphic, the image of the open neighborhood is an (almost) complex submanifold of $M$. Hence, it is $\C^{\rk(\phi)}$. In particular, the image $\phi(X\setminus A_r)$ is a complex analytic variety. Then Lemma \ref{imageprop} follows from a  theorem of Shiffman \cite{Sh} (see also \cite{Ki}) by taking $k=\rk(\phi)$.

\begin{thm}[Shiffman]\label{Shiff}
Let $U$ be open in $\C^n$ and let $E$ be closed in $U$. Let $X$ be a pure $k$-dimensional complex analytic subvariety in $U\setminus E$, and let $\bar X$ be the closure of $X$ in $U$. If $\mathcal H^{2k-1}(E)=0$ then $\bar X$ is a pure $k$-dimensional complex analytic subvariety in $U$.
\end{thm}

\end{proof}

If we study a rational map from a closed almost complex manifold $(X, J_X)$ to a complex manifold $(M, J)$, {\it i.e.} a proper pseudoholomorphic map $\phi: X\setminus B\rightarrow Y$ where $B$ has Hausdorff codimension at least two, we know the image $\phi(X\setminus B)$ has the structure of analytic subvariety locally near any point in the image following from the argument of Proposition \ref{imageprop}. By abuse of notation, we call the image an open analytic subvariety.

\subsection{Pseudoholomorphic structures}
In this subsection, we recall the one-to-one correspondence of  pseudoholomorphic structures and bundle almost complex structures on a complex vector bundle $E$ over the almost complex manifold $(X, J)$ as established in \cite{CZ, DeT}.

A {\it pseudoholomorphic structure} on $E$ is given by a differential operator $\bar{\partial}_E: \Gamma(X, E)\rightarrow  \Gamma(X, (T^*X)^{0,1}\otimes E)$ which satisfies the Leibniz rule $$\bar{\partial}_E (fs)=\bar{\partial}f\otimes s+f\bar{\partial}_Es$$ where $f$ is a smooth function and $s$ is a section of $E$. A particularly important pseudoholomorphic structure is the natural one $\bar{\partial}_m: \Gamma(X, \mathcal K_X^{\otimes m})\rightarrow (T^*X)^{0,1}\otimes \Gamma(X, \mathcal K_X^{\otimes m})$ for $ m\geq 2$ on the pluricanonical bundle  $\mathcal K_X^{\otimes m}$ induced from the $\bar{\partial}$ operator inductively by the product rule
$$\bar{\partial}_m(s_1\otimes s_2)=\bar{\partial}s_1\otimes s_2+s_1\otimes \bar{\partial}_{m-1} s_2.$$

For any Hermitian bundle $(E, h_E)$ with a pseudoholomorphic structure $\bar{\p}_E$, there is a unique Hermitian connection $\nabla$ so that its $(0,1)$-component $\nabla^{(0,1)}=\bar{\partial}_E$ (see Lemma 3.3 in \cite{CZ}).

We can define a unique dual pseudoholomorphic structure $\bar{\p}_{E^*}: \Gamma(X, E^*)\rightarrow \Gamma(X, (T^*X)^{0,1}\otimes E^*)$ on the dual bundle $E^*$: for any section $s^*\in \Gamma(X, E^*)$ and any section $s'\in \Gamma(X, E)$, let
\begin{align}\label{L} (\bar{\p}_{E^*}(s^*))(s')=\bar{\p} (s^*(s'))-s^*(\bar{\p}_E(s')). \end{align}
It is easy to verify that $\bar{\p}_{E^*}$ satisfies the Leibniz rule, giving a pseudoholomorphic structure.

On the other hand, a {\it bundle almost complex structure} as in \cite{DeT} is an almost complex structure $\mathcal J$ on $TE$ so that
\begin{enumerate}[(i)]
 \item the projection is $(\mathcal J, J)$-holomorphic,
 \item  $\mathcal J$ induces the standard complex structure on each fiber, i.e. multiplying by $i$,
 \item  the fiberwise addition $\alpha: E\times_M E\rightarrow E$ and the fiberwise multiplication by a complex number $\mu:\C\times E\rightarrow E$ are both pseudoholomorphic.
 \end{enumerate}

It is shown in \cite{DeT} that a bundle almost complex structure $\mathcal J$ on $E$ determines a pseudoholomorphic structure $\bar{\p}_{\mathcal J}$, and the map $\mathcal J\mapsto  \bar{\p}_{\mathcal J}$ is a bijection between the spaces of bundle almost complex structures and pseudoholomorphic structures on $E$. Furthermore, it is shown as Corollary 2.4 in \cite{CZ} that  $\bar{\p}_{\mathcal{J}}s=0$ if and only if $s$ is $(J, \mathcal{J})$ holomorphic for any $s\in \Gamma(X, E)$. We denote this space of pseudoholomorphic sections by $H^0(X, (E, \mathcal J))$. In particular, the space of pseudoholomorphic sections of $(\mathcal K_X^{\otimes m}, \bar{\partial}_m)$ is denoted by
$H^0(X, \mathcal K_X^{\otimes m})$.

These two equivalent ways of description of a pseudoholomorphic section have their respective use. The PDE description was used to show the dimension of pseudoholomorphic sections is finite, while the geometric description bridges to the intersection theory of almost complex submanifolds \cite{Z2}. In particular, when $E$ is a complex line bundle over $4$-manifold $(X, J)$, the zero locus of a pseudoholomorphic section $s$ is a $J$-holomorphic $1$-subvariety in class $c_1(E)$.

\subsection{Pluricanonical maps}
One can still define pluricanonical maps for an almost complex manifold $(X, J)$. More generally, for any complex line bundle $L$ with a bundle almost complex structure $\mathcal J$, $\Phi_{L, \mathcal J}: X\rightarrow \mathbb CP^{N}$ is defined as $\Phi_{L, \mathcal J}(x)=[s_0(x): \cdots :s_{N}(x)]$, where $s_i$ constitute a basis of the linear space $H^0(X, (L, \mathcal J))$.
When $L=\mathcal K_X^{\otimes m}$ and $\mathcal J=\mathcal J_J$ induced by $\bar{\p}_m$, $\Phi_{L, \mathcal J}$ is denoted by $\Phi_m$ and called the pluricanonical map. This map is well defined because the transition map for a complex line bundle takes value in $\text{GL}(1, \mathbb C)$. Hence, changing the local coordinates would lead to a common multiple of all $s_i$.

Furthermore, $\Phi_{L, \mathcal J}$ is a pseudoholomorphic map from $(X, J)$ to $(\mathbb CP^{N}, J_{std})$. It can be seen from each map $\Phi_{L, \mathcal J}^{ij}(x)= \frac{s_j}{s_i}(x)$. Each $\Phi_{L, \mathcal J}^{ij}$ is a section of the trivial bundle over $X\setminus \{s_i^{-1}(0)\}$.
The map $\frac{1}{s_i(x)}$  is a pseudoholomorphic section of the dual bundle $L^*$ over $X\setminus \{s_i^{-1}(0)\}$ with respect to the  pseudoholomorphic structure on $L^*$ dual to $\bar{\p}_{\mathcal J}$, by applying \eqref{L} to the relation $1=s_i\cdot \frac{1}{s_i}$.  Apply \eqref{L} again, $\Phi_{L, \mathcal J}^{ij}$ is a pseudoholomorphic section of the trivial bundle over $X\setminus \{s_i^{-1}(0)\}$. Hence, it implies that $\Phi_{L, \mathcal J}: X\setminus B\rightarrow \mathbb CP^N$ is a pseudoholomorphic map. Here $B\subset X$ is the set of the base locus, {\it i.e.} the set of points $x\in X$ such that $s_i(x)=0$ for all $0\le i\le N$.

We are able to use the image $\Phi_L(X\setminus B)$, in particular its dimension, to study the almost complex manifold $(X, J)$. By the argument of Proposition \ref{imageprop}, we know the image is an open analytic subvariety. Apparently, $\dim_{\C}\Phi_L(X\setminus B)\le \dim_{\C} X$. Moreover, we have the following.

\begin{thm}\label{Koddimmax}
Let $(X, J)$ be a connected almost complex manifold and $L$ be a complex line bundle with a bundle almost complex structure. If $\dim_{\C}\Phi_L(X\setminus B)=\dim_{\C} X$, then $J$ is integrable. If $P_L(X, J)=\dim H^0(X, (L, \mathcal J))\ge 2$, then $\dim_{\C}\Phi_L(X\setminus B)\ge 1$.
\end{thm}
\begin{proof}
If $\Phi_L(X\setminus B)$ is the set of finitely many points, then we know that for any two pseudoholomorphic sections $s_1, s_2$ of $L$, we have $s_2-ks_1=0$ in an open subset of $X$ for some constant $k$. However, it implies $s_2-ks_1\equiv 0$ by Proposition 2.3 of \cite{Z2}. Hence if $P_L(X, J)\ge 2$,  the image has $\dim_{\C}\Phi_L(X\setminus B)\ge 1$.

When $\dim_{\C}\Phi_L(X\setminus B)\ge 1$, by the unique continuation for pseudoholomorphic submanifolds,  Proposition 2.3 of \cite{Z2}, both the base locus $B$ and the singularity subset of $\Phi_L$ cannot contain an open subset of $X$. Moreover, since the singularities of the image $\Phi_L(X\setminus B)$ has complex codimension at least $1$ in $\Phi_L(X\setminus B)$, its preimage also does not contain an open subset of $X$. We now look at the complement of these subsets of $X$, which is called the regular part $X_{reg}$ by abuse of terminology.

When $\dim_{\C}\Phi_L(X\setminus B)=2$, we know the pseudoholomorphic map $\Phi_L$ maps $X_{reg}$ diffeomorphically to its image. Hence $J|_{X_{reg}}$ is integrable, which in turn implies $J$ is integrable on whole $X$ since $X=\overline{X_{reg}}$.
\end{proof}

As a result, we can define another version of Iitaka dimension for almost complex manifolds.

\begin{defn}
The Iitaka dimension $\kappa_J(X, (L, \mathcal J))$ of a complex line bundle $L$ with bundle almost complex structure $\mathcal J$ over $(X, J)$ is defined as
$$\kappa_J(X, (L, \mathcal J))=\begin{cases}\begin{array}{cl} -\infty, &\ \text{if} \ P_{L^{\otimes m}}=0\  \text{for any} \ m\geq 0\\
\max\dim \Phi_{L^{\otimes m}} (X\setminus B_{m, L, \mathcal J}), &\  \text{otherwise.}\end{array}
\end{cases}$$

The Kodaira dimension $\kappa_J(X)$ is defined by choosing $L=\mathcal K_X$ and $\mathcal J$ to be the bundle almost complex structure induced by $J$.
\end{defn}

In Theorem 5.3 of \cite{CZ}, we have proved that there is a one-to-one correspondence between elements of $H^0(X, \mathcal K_X^{\otimes m})$ and $H^0(Y, \mathcal K_Y^{\otimes m})$ when there is a pseudoholomorphic degree one map between closed almost complex $4$-manifolds $(X, J)$ and $(Y, J_Y)$. Moreover, these elements coincide at an open subset of each $X$ and $Y$ whose complements are proper subvarieties of $X$ and $Y$ respectively. Hence, it follows that $\kappa_J(X)$ is a birational invariant.

\begin{thm}\label{Kodbir'}
Let $u: (X, J)\rightarrow (Y, J_Y)$ be a degree one pseudoholomorphic map between closed almost complex $4$-manifolds. Then  $\kappa_{J}(X)= \kappa_{J_Y}(Y)$.
\end{thm}

It essentially follows from Theorem 6.10 in \cite{CZ} that all the values $\{-\infty, 0, 1, \cdots, n-1\}$ could be realized by $\kappa_J(X)$ when $(X, J)$ is non-integrable. It is clear for values $-\infty$ and $0$. For examples with $\kappa_J\ge 1$, we look at the non-integrable almost complex structures on $T^2\times S$ in Section 6.3 of \cite{CZ}. By Lemma 6.5 in  \cite{CZ}, the pluricanonical maps are base-point-free and have the same image as that of a Riemann surface $S$ with $g\ge 2$. In particular, it has dimension $1$.  By taking product of this $T^2\times S$ with other Riemann surfaces $T^2$ and $\Sigma_g$, we know the image of pluricanonical map is the product of Riemann surfaces by the K\"unneth formula, Proposition 6.8 in \cite{CZ}. This gives us all the possible values of $\kappa_J$.

Below is a structural result for almost complex $4$-manifolds with $\kappa_J=1$. It shows that these manifolds are in fact generalizations of elliptic surfaces.

\begin{thm}\label{ellsur}
Let $(X, J)$ be an almost complex $4$-manifolds admitting a base-point-free pluricanonical map.
If $\kappa_J(X)=1$, then $X$ admits a pseudoholomorphic fibration with each fiber an elliptic curve and finitely many singular fibers.
\end{thm}
\begin{proof}
Let $\Phi_m$ be the base-point-free pluricanonical map. Then by Proposition \ref{imageprop}, the image $\Phi_m(X)$ is a projective curve in a projective space.
Notice the intersection number argument to prove Theorem 1.2 of \cite{Z2} still works when the target space is a pseudoholomorphic curve. Hence the preimage $\Phi_m^{-1}(p)$, $p\in \Phi_m(X)$, is a $J$-holomorphic $1$-subvariety. Then $\Phi_m(X)$ is in fact a Riemann surface since otherwise at a singular point $p\in \Phi_m(X)$, a point where the fiber $\Phi_m^{-1}(p)$ is smooth is a singular point of $X$.

By the argument of Proposition \ref{comsubvar}, the set of critical values of the map $\Phi_m$ is an analytic subvariety of $\Phi_m(X)$ of (real) codimension at least $2$. In our situation, it has to be a set of finitely many points.

For a fiber $C_p=\Phi_m^{-1}(p)$, it is the zero locus of a pluricanonical section. Hence $[C_p]=mK_X$. Moreover, because $\Phi_m$ is base-point-free, any two fibers are disjoint. Hence, $[C_p]^2=0=K_X\cdot [C_p]$. By adjunction formula, $C_p$ is a $J$-holomorphic $1$-subvariety whose connected components are possibly multiple elliptic curves. Moreover, for generic $p$, $C_p$ is a disjoint union of smooth elliptic curves.

We look at the collection of all connected components $C_{p, i}$ of all the fibers $C_p$. It is given the natural topology by Hausdorff convergence. It could also be endowed with a natural Riemann surface structure. Without loss of generality, we can assume a generic fiber has multiplicity one. If a connected component $C_{p, i}$ has multiplicity one, then choose a holomorphic disk intersect it transversely at a smooth point. It will intersect at nearby fibers $C_{p'}$ uniquely at one point and hence it is locally homeomorphic to $\Phi_m(X)$ near $p$. If the connected component $C_{p, i}$ has multiplicity $l>1$, we know for any nearby $q$, there will be $l$ smooth elliptic curves in $C_q$ approaching to $(|C_{p, i}|, l)$ since critical value $p$ is isolated. Hence any disk intersecting $C_{p,i}$ transversely at a smooth point will also intersect transeversely the $l$ smooth elliptic curves in $C_q$ approaching to $(|C_{p, i}|, l)$ for any nearby $q$.
Hence, the local model is a  covering of disk over disk branched at origin, {\it i.e.} $z\mapsto z^l$.

In summary, we have a Riemann surface $\Sigma$ and a branched covering $\pi: \Sigma\rightarrow \Phi_m(X)$ such that the pluricanonical map $\Phi_m$ is factored through $\Sigma$ as $\pi\circ \Phi_m'$. Moreover, each fiber of the pseudoholomorphic map $\Phi_m': X\rightarrow \Sigma$ is a (connected) elliptic curve.
\end{proof}

This result is complementary to Theorem 6.7 in \cite{CZ}. We expect to classify almost complex $4$-manifolds with $\kappa_J(X)=\kappa^J(X)=1$.

The two versions $\kappa^J$ and $\kappa_J$ are equal to each other and equal to the original Iitaka dimension when $J$ is integrable and $L$ is a holomorphic line bundle.
\begin{que}\label{k=k}
Is $\kappa_J(X, (L, \mathcal J))=\kappa^J(X, (L, \mathcal J))$ for a complex line bundle $L$ with bundle almost complex structure $\mathcal J$ over $(X, J)$?
\end{que}

In general, we even do not know whether $\kappa^J$ is an integer. In dimension $4$, Question \ref{k=k}  is equivalent to: if $\kappa_J(X, (L, \mathcal J))=1$, is it true that $\kappa^J(X, (L, \mathcal J))=1$?

As in the complex setting, we could also study the canonical model of an almost complex manifold by the graded ring construction. When the $(L, \mathcal J)$ sections $H^0(X, (L, \mathcal J))$ is base-point-free, we know $\Phi_{L, \mathcal J}(X)$ is an analytic subvariety by Proposition \ref{imageprop} and hence a projective subvariety by Chow's Theorem.

When $H^0(X, (L, \mathcal J))$ has base points, we are able to show that it is still true in dimension $4$.

\begin{prop}
\label{pcm}
Let $\phi$ be a pseudoholomorphic map from $X\setminus B$ to a closed complex manifold $(M, J)$, where $(X, J_X)$ is a closed almost complex $4$-manifold, $(M, J)$ a closed complex manifold, $B\subset X$ a set which is the intersection of  $\psi_i^{-1}(Z_i)$ for pseudoholomorphic maps $\psi_i: X\rightarrow M_i$ where $Z_i$ is an almost complex submanifold of $M_i$. Then $\overline{\phi(X\setminus B)}$ is an analytic subvariety of $M$.
\end{prop}
\begin{proof}
When $\dim_{\C} \phi(X\setminus B)=0$, since $X\setminus B$ is connected, it will be mapped to a single point. Thus its closure is itself and there is nothing to prove.

When $\dim_{\C} \phi(X\setminus B)=2$, we know the pseudoholomorphic map $\phi$ maps $X_{reg}$ diffeomorphically to its image. Hence $J|_{X_{reg}}$ is integrable, which in turn implies $J$ is integrable on whole $X$ since $X=\overline{X_{reg}}$. Hence, it is reduced to the classical complex setting, and the closure $\overline{\phi(X\setminus B)}$ is an analytic subvariety, thus a projective subvariety of $\mathbb CP^N$. Moreover, the set $\overline{\phi(X\setminus B)}\setminus \phi(X\setminus B)$ is a collection of subvarieties.

When $\dim_{\C} \phi(X\setminus B)=1$, we need to show $\overline{\phi(X\setminus B)}$ is an analytic curve in $\mathbb CP^N$. We have shown $\phi(X\setminus B)$ is an open analytic curve. If  $\overline{\phi(X\setminus B)}\setminus \phi(X\setminus B)\ne \emptyset$, choose a point $p$ from it. Suppose a sequence of points $p_i$ approach to $p$. Since critical values are isolated points, we can assume $p_i$ are regular value. Hence, the inverse image $\phi^{-1}(p_i)$ are almost complex submanifolds of $X\setminus B$. Since $p\in \overline{\phi(X\setminus B)}\setminus \phi(X\setminus B)$, all limiting points of these fibers $\phi^{-1}(p_i)$ belong to $B$. Since $B$ is a union of pseudoholomorphic $1$-subvarieties and finitely many points by Theorem 1.2 of \cite{Z2} and positivity of intersection of pseudoholomorphic curves, we can choose a pseudoholomorphic disk that intersects $B$ only at such a limiting point. Because $p\in \overline{\phi(X\setminus B)}\setminus \phi(X\setminus B)$, the disk will be mapped to a punctured disk in $\phi(X\setminus B)$ around $p$. In particular, it implies the set $\overline{\phi(X\setminus B)}\setminus \phi(X\setminus B)$ has Hausdorff dimension $0$.

Now $\overline{\phi(X\setminus B)}$ defines a rectifiable current in $M$. The boundary $\partial(\overline{\phi(X\setminus B)})\subset K\cup A$, where $A$ is the set of critical points and $K=\overline{\phi(X\setminus B)}\setminus \phi(X\setminus B)$. Since $\mathcal H^{1}(K\cup A)=0$, we know $\partial(\overline{\phi(X\setminus B)})=0$ by Support Theorem (\cite{Fed} p. 378). Hence, $\overline{\phi(X\setminus B)}$ is an analytic curve of $M$ by Theorem 5.2.1 of \cite{Ki}.
\end{proof}

We believe that a more careful analysis of the neighborhood of the base locus and the points in the closure would lead us to the same statement for higher dimensions.

\begin{cor}\label{proj}
Let $B$ be the set of base locus of a complex line bundle $(L, \mathcal J)$ over an almost complex $4$-manifold $(X, J)$. Then the closure $\overline{\Phi_{L, \mathcal J}(X\setminus B)}$ is a projective subvariety of $\mathbb CP^N$, where $N=\dim H^0(X, (L, \mathcal J))$.
\end{cor}
\begin{proof}
It is an analytic subvariety of $\mathbb CP^N$ by Proposition \ref{pcm}, and hence a projective subvariety by Chow's Theorem.
\end{proof}

As remarked above, we expect $\overline{\Phi_{L, \mathcal J}(X\setminus B)}$ to be a projective subvariety of $\mathbb CP^N$, where $N=\dim H^0(X, (L, \mathcal J))$ for any almost complex manifold $(X, J)$.

Given a complex line bundle $L$ with a bundle almost complex structure $\mathcal J$. It determines a pseudoholomorphic structure $\bar{\p}_{\mathcal J}$ and thus induces that on $L^{\otimes m}, m>0$. For any two complex line bundles $L_1, L_2$ with bundle almost complex structures $\mathcal J_1, \mathcal J_2$, we can define a pseudoholomorphic structure on $L_1\otimes L_2$ by Leibniz rule. In turn, it determines a bundle almost complex structure $\mathcal J_1\otimes \mathcal J_2$ on $L_1\otimes L_2$. Moreover, pseudoholomorphic sections $s_1, s_2$ of $(L_1, \mathcal J_1)$ and $(L_2, \mathcal J_2)$ give rise to pseudoholomorphic section $s_1\cdot s_2$.  We define the section ring of $(L, \mathcal J)$ to be  $$R(X, L, \mathcal J)=\oplus_{d=0}^{\infty}H^0(X, L^{\otimes d}).$$
When $L=\mathcal K_X$ and $\mathcal J=\mathcal J_J$, we will call the section ring the canonical ring and the corresponding scheme $Proj(R)$ the canonical model of $X$.

When $(X, J)$ is a complex manifold, the Iitaka dimension is also equal to the  transcendental degree of the quotient field of the section ring. For projective varieties,  the canonical ring is finitely generated \cite{BCHM}.
Moreover, it is isomorphic to the closure of the image of $\Phi_{m}$ when $m$ is large. We expect these also hold for a general almost complex manifold $(X, J)$. However, we remark that  the section ring need not to be finitely generated even for holomorphic line bundles over projective manifolds.

\section{Bounding $\kappa^J$ by symplectic Kodaira dimension}\label{Weil}
We have seen  that the Iitaka dimensions of a non-integrable almost complex $4$-manifold are usually very small. It could be interpreted as we do not have too many Cartier divisors for a non-integrable almost complex structure. Meanwhile, we can also look at the growth rate of the dimension of the space of Weil divisors, {\it i.e.} the moduli space of the $J$-holomorphic subvarieties $\mathcal M_{me}$ in positive multiple of a class $e\in H^2(M, \mathbb Z)$.  In this section, we will give another generalization of Iitaka dimension by assuming the almost complex structure is tamed. We will show this corresponding version of Kodaira dimension is in general greater than the previous two versions $\kappa^J$ and $\kappa_J$, but is bounded above by the symplectic Kodaira dimension.

When $e$ is a spherical class, we have shown in \cite{s2mod} that, for any almost complex structure, $\mathcal M_e$ behaves like linear system in algebraic geometry. However, for a general class $e$ and a general almost complex structure, the moduli space is very complicated. We suggest the following way to estimate (the upper bound of) its dimension.  Recall there is a natural topology on $\mathcal M_e$ defined in terms of convergent sequences (see {\it e.g.} \cite{T09}). A sequence $\{\Theta_k\}$ in $\mathcal M_e$ converges to a given element $\Theta$ if the following two conditions are met:

\begin{itemize}
\item  $\lim_{k\to \infty} (\sup _{z\in |\Theta|} \hbox{dist}(z, |\Theta_k|)
+\sup _{z'\in |\Theta_k|} \hbox{dist}(z', |\Theta|))=0$.

\item  $\lim_{k\to \infty} \sum_{(C', m')\in \Theta_k} m'\int_{C'}\nu=\sum_{(C, m)\in \Theta} m\int_{C}\nu$ for any given smooth 2-form $\nu$.
\end{itemize}

 If for any element $x\in \mathcal M_e$, there is a open neighborhood $U\subset \mathcal M_e$ and a continuous map from a complex manifold $\mathcal X_U$ onto $U$, we would say the local dimension of $\mathcal M_e$ at $x$ is bounded above by $\dim_{\C} \mathcal X_U$. 
The manifold $\mathcal X_U$ is called a local parameter space. We can define $\dim \mathcal M_e$ to be the infimum of the dimensions of all the possible parameter spaces.

We can thus define $$\kappa_W^J(X, e)=\begin{cases} \begin{array}{cl} -\infty, &\ \text{if} \ \mathcal M_{me}=\emptyset\  \text{for all} \ m\geq 0\\
\limsup_{m\rightarrow \infty} \dfrac{\log (1+\dim \mathcal M_{me})}{\log m}, &\  \text{otherwise.}\end{array}
\end{cases}$$

When $e$ is the canonical class, we will simply write $\kappa_W^J(X)$ for this Weil divisor version of Kodaira dimension.

Apparently, we have
\begin{lem}\label{0tosec}
When $(X, J)$ is a $4$-dimensional tamed almost complex manifold, $\kappa^J(X, (L, \mathcal J))\le \kappa_W^J(X, c_1(L))$.
\end{lem}
\begin{proof}
It follows from two facts. First, the zero locus of sections $s\in H^0(X, (L^{\otimes m}, \mathcal J))$ is a $J$-holomorphic subvariety in class $mc_1(L)$ by Corollary 1.3 of \cite{Z2}.

Second,  if two non-trivial pseudoholomorphic sections $s_1, s_2$ of $(L, \mathcal J)$ have the same zero locus $Z$ (counted with multiplicities), then $s_1=ks_2$ for a nonzero constant $k$. If it is not true, then there is a point $x\notin Z$ such that $s_1(x)=ks_2(x)$ for some constant $k\ne 0$, but $s_1\ne ks_2$. By the definition of bundle almost complex structure, $s_1-ks_2$  is also a pseudoholomorphic section of $(L, \mathcal J)$. However, its zero locus contains $Z\cup \{x\}$. In particular, the homology class of the zero locus is different from $c_1(L)$ since any $J$-holomorphic subvariety has non-trivial homology when $J$ is tamed. This contradiction implies that the zero locus of $(L, \mathcal J)$ determines the pseudoholomorphic section up to a nonzero constant.

These two facts together imply that $\mathbb PH^0(X, (L^{\otimes m}, \mathcal J))$ is embedded in $\mathcal M_{mc_1(L)}$. Hence $\dim \mathcal M_{mc_1(L)}\ge P_{L^{\otimes m}, \mathcal J}$, and thus $\kappa^J(X, (L, \mathcal J))\le \kappa_W^J(X, c_1(L))$.
\end{proof}

In general, $\kappa^J_W$ is greater than $\kappa^J$.

\begin{prop}\label{bigKod2}
Suppose $(X, J)$ is a tamed almost complex $4$-manifold with $b^+=1$. Let $e\in H^2(X, \mathbb Z)$ be a class in the positive cone with $e\cdot e>0$. Then $\kappa^J_W(X, e)\ge 2$.
\end{prop}

Meanwhile, $\kappa_J\le 1$ for non-integrable $J$ as shown in Theorem \ref{Koddimmax}.

\begin{proof}
We let $\iota_{e}=e\cdot e-K_J\cdot e$.
By Seiberg-Witten wall crossing formula and SW=Gr \cite{LLimrn, Tbk}, then for any sufficiently large integer $n$, there is an element in $\mathcal M_{ne}$ passing through any given $\frac{1}{2}\iota_{ne}$ points if $b_1\ne 2$. If $b_1=2$, there is an element in $\mathcal M_{ne}$ passing through any given $\frac{1}{2}\iota_{ne}-1$ points and intersecting each of the chosen generators of $H_1(X, \mathbb Z)$/torsion (see {\it e.g.} Proposition 1.1 in \cite{T17}).

We first choose an element in $\mathcal M_{ne}$, say $\Theta$. Choose  $\frac{1}{2}\iota_{ne}$ points $x_i$ (or $\frac{1}{2}\iota_{ne}-1$ points when $b_1=2$) on the smooth part of $|\Theta|$. Choose the same number of ($J$-holomorphic) disks $D_i$ containing these points and transverse to $|\Theta|$. We assume $D_i\cap |\Theta|=x_i$. When we choose the disks small, the corresponding elements in $\mathcal M_{ne}$ passing through one point in each disk would intersect $D_i$ only once. Hence, the elements determined by elements in $D_1\times \cdots \times D_{\frac{1}{2}\iota_{ne}}$ (or $D_1\times \cdots \times D_{\frac{1}{2}\iota_{ne}-1}$ and two circles) are different. Hence, $\dim \mathcal M_{ne}\ge \frac{1}{2}\iota_{ne}-1$. This implies $\kappa^J_W(X, e)\ge 2$.
\end{proof}

If $e$ is the class of a symplectic form tamed by $J$, then by the argument of Theorem \ref{W=s}, we have $\kappa^J_W(X, e)= 2$.

We can bound $\kappa_W^J(X)$ from above by the symplectic Kodaira dimension \cite{L}, whose definition is recalled in the following.

\begin{defn}\label{symKod}
For a minimal symplectic $4$-manifold $(X,\omega)$ with symplectic
canonical class $K_{\omega}$,   the Kodaira dimension of
$(X,\omega)$
 is defined in the following way:
$$
\kappa^s(X,\omega)=\begin{cases} \begin{array}{cll}
-\infty & \hbox{ if $K_{\omega}\cdot [\omega]<0$ or} & K_{\omega}\cdot K_{\omega}<0,\\
0& \hbox{ if $K_{\omega}\cdot [\omega]=0$ and} & K_{\omega}\cdot K_{\omega}=0,\\
1& \hbox{ if $K_{\omega}\cdot [\omega]> 0$ and} & K_{\omega}\cdot K_{\omega}=0,\\
2& \hbox{ if $K_{\omega}\cdot [\omega]>0$ and} & K_{\omega}\cdot K_{\omega}>0.\\
\end{array}
\end{cases}
$$

The Kodaira dimension of a non-minimal manifold is defined to be that of any of its minimal models.
\end{defn}

Here $K_{\omega}$ is defined as the first Chern class of the cotangent bundle for any almost complex structure compatible with $\omega$.
We remark that $\kappa^s$ is independent of the symplectic form $\omega$ and thus a smooth invariant.

First, we need a lemma. This is an adaption of Lemma 2.5 in \cite{s2mod}. Here, we no longer assume $e$ is $J$-nef and $\Theta\in \mathcal M_e$ is connected. However, we fix more points on each irreducible components.

\begin{lem}\label{pointscurve}
Let $J$ be an almost complex structure on a compact smooth $4$-manifold $X^4$ and $\{(x_1, d_1), \cdots, (x_k, d_k)\}$ are points with weight.
Let $\Theta=\{(C_i, m_i), 1\le i\le n\}\in \mathcal M_e$ be a subvariety passing through these points with weight such that there are at least $m_i\max\{e\cdot e_{C_i}+1, 0\}$ points (counted wtih multiplicities) on $C_i$ for each $i$ and all $x_i$ are smooth points. Then there are no other such subvariety in class $e$.
\end{lem}
\begin{proof}
Suppose there is another subvariety $\Theta'\in \mathcal M_e$ passing through this set of points with multiplicities.

We rewrite two subvarieties $\Theta, \Theta' \in \mathcal M_e$, allowing $m_i=0$ in the notation, such that they share the same set of irreducible components formally, i.e. $\Theta=\{(C_i, m_i)\}$ and $\Theta'=\{(C_i, m'_i)\}$. Then for each $C_i$, if $m_i\le m'_i$, we change the components to $(C_i, 0)$ and $(C_i, m'_i-m_i)$. At the same time, if a point $x$, as one of $x_1, \cdots, x_k$, is on $C_i$, then the weight is reduced by $m_i$ as well. Similar procedure applies to the case when $m_i> m_i'$.  Apply this process to all $i$ and discard finally all components with multiplicity $0$ and denote them by $\Theta_0,\Theta'_0$ and still use $(C_i, m_i)$ and $(C_i, m'_i)$ to denote their components. Notice they are homologous, formally having homology class $$e-\sum_{m_{k_i}<m_{k_i}'} m_{k_i}e_{C_{k_i}}-\sum_{m_{l_j}'<m_{l_j}} m'_{l_j}e_{C_{l_j}}-\sum_{m_{q_p}'=m_{q_p}} m'_{q_p}e_{C_{q_p}}.$$

There are two ways to express the class, by taking $e=e_{\Theta}$ or $e=e_{\Theta'}$ in the above formula. Namely, it is $$\sum_{m_{k_i}<m_{k_i}'}(m_{k_i}'-m_{k_i})e_{C_{k_i}}+\hbox{others}=e_{\Theta_0'}=e_{\Theta_0}=\sum_{m_{l_j}'<m_{l_j}} (m_{l_j}-m'_{l_j})e_{C_{l_j}}+\hbox{others}.$$
Here the term ``others" means the terms $m_ie_{C_i}$ or $m_i'e_{C_i}$ where $i$ is not taken from $k_i$, $l_j$ or $q_p$.

Now $\Theta_0$ and $\Theta_0'$ have no common components. By the process we just applied, counted with weight, there are  at least $e\cdot e_{\Theta_0}+\sum_{i: C_i\in \Theta_0}m_i$ points on $\Theta_0$. These points are also contained in $\Theta_0'$ with right weights. Hence $\Theta_0$ and $\Theta_0'$ would intersect at least $e\cdot e_{\Theta_0}+\sum_{i: C_i\in \Theta_0}m_i$ points with weight.

We notice that $e\cdot e_{\Theta_0}\ge e_{\Theta_0}\cdot e_{\Theta_0'}\ge 0$. In fact, the difference $e-e_{\Theta_0}=e-e_{\Theta_0'}$ has $3$ types of terms, any of them pairing non-negatively with the class $e_{\Theta_0}$. For the terms with index $k_i$, {\it i.e.} the terms with $m_{k_i}<m_{k_i}'$, we use the expression of $e_{\Theta_0}=\sum_{m_{l_j}'<m_{l_j}} (m_{l_j}-m'_{l_j})e_{C_{l_j}}+\hbox{others}$ to pair with. Since the irreducible curves involved in the expression are all different from $C_{k_i}$, we have $e_{C_{k_i}}\cdot e_{\Theta_0}\ge 0$. Similarly, for $C_{l_j}$, we use the expression of $e_{\Theta_0'}=\sum_{m_{k_i}<m_{k_i}'}(m_{k_i}'-m_{k_i})e_{C_{k_i}}+\hbox{others}$. We have  $e_{C_{l_j}}\cdot e_{\Theta_0'}\ge 0$. For $C_{q_p}$, we could use either $e_{\Theta_0}$ or $e_{\Theta_0'}$. Since $e_{\Theta_0}=e_{\Theta_0'}$, we have $(e-e_{\Theta_0})\cdot e_{\Theta_0}\ge 0$.

Since $\Theta\ne \Theta'$, we know $\sum_{i: C_i\in \Theta_0}m_i>0$. Hence $e\cdot e_{\Theta_0}+\sum_{i: C_i\in \Theta_0}m_i>e_{\Theta_0}\cdot e_{\Theta_0'}$. This inequality  $e\cdot e_{\Theta_0}> e_{\Theta_0}^2$ implies there are more intersections  than the homology intersection number $e_{\Theta_0}^2$ of our new subvarieties $\Theta_0$ and $\Theta_0'$. This contradicts to the local positivity of intersection and the fact that $\Theta_0, \Theta_0'$ have no common component. The contradiction implies that $\Theta$ is the unique such subvariety as described in the statement.
\end{proof}

We are ready to prove the main result of this section.

\begin{thm}\label{W<s}
Let $(X, J)$ be a tamed almost complex $4$-manifold. Then $$\kappa_W^J(X)\le \kappa^s(X).$$
\end{thm}
\begin{proof}
For any pseudoholomorphic section $s\in H^0(X, \mathcal K_X^{\otimes m})$, it is a $(J, \mathcal J_{J})$-holomorphic map from $X$ to the total space of $\mathcal K_X^{\otimes m}$ endowed with the standard bundle almost complex structure induced by $\bar{\p}_m$. Hence, by Corollary 1.3 of \cite{Z2}, we know the zero locus is a $J$-holomorphic $1$-subvariety in class $mK_X$.

When $\kappa^s(X)=-\infty$, the canonical class is not effective, so $\mathcal M_{mK_X}=\emptyset, \forall m>0$. Hence, $\kappa^J_W(X)=-\infty$.

When $\kappa^s(X)=0$, then $K_X=K_{X_{\min}}+\sum E_i$, where $nK_{X_{\min}}=0$ for $n=1$ or $2$. Moreover, by Proposition \ref{1pcE}, there is a unique $J$-holomorphic representative for positive linear combination of $-1$ classes, for any tamed $J$. This shows $\mathcal M_{mK_X}$ is a single point. Hence, $\kappa_W^J(X)= 0=\kappa^s(X)$.

When $\kappa^s(X)=1$ or $2$, we know $\mathcal M_{K_X}\ne \emptyset$ when $b^+>1$ and $\mathcal M_{2K_X}\ne \emptyset$ when $b^+=1$. Hence, we will estimate the elements in $\mathcal M_{mK_X}$ or $\mathcal M_{2mK_X}$ respectively. For simplicity of notation, we will only deal with the case when $\mathcal M_{K_X}\ne \emptyset$ and estimate $\mathcal M_{mK_X}$. The other case ({\it i.e.} when $\mathcal M_{2K_X}\ne \emptyset$) is almost identical.

When $\kappa^s(X)=1$, $K_X=K_{X_{\min}}+\sum E_i$, where $K_{X_{min}}^2=0$. Let us first assume $X$ is minimal. Then $K_X=K_{X_{min}}$ is $J$-nef by Lemma 2.1 of \cite{p=h}. Hence $K_X\cdot [C_i]=0$ for any irreducible component of a $J$-holomorphic subvariety $\Theta=\{(C_i, m_i)\}$ in class $mK_X$. In particular, each connected component of $\Theta$ has square $0$. Moreover, by Theorem 1.4 of \cite{LZrc}, each irreducible component is a torus with square zero or a $-2$ sphere.

In fact, choose a $J$-tame symplectic form $\omega$ with integral cohomology class, we know there are at most $mC_{K_X}:=m\lceil\frac{[\omega]\cdot K_X}{\inf_{C} [\omega]\cdot [C]}\rceil$ irreducible components (counted with multiplicities) for any elements in $\mathcal M_{mK_X}$. Here, we take the infimum over all non-trivial $J$-holomorphic curves $C$, which is non-zero by our integrity assumption. For any $\Theta\in \mathcal M_{mK_X}$, choose a point on the smooth part of each irreducible component $C_i$ with multiplicity $m_i$, we know there is a unique $J$-holomorphic subvariety in $\mathcal M_{mK_X}$ passing through them, by Lemma \ref{pointscurve}.  Hence, $X^{mC_{K}}$ contains a parameter space of $\mathcal M_{mK_X}$. Since the dimension varies linearly, this implies $\kappa^J_W(X)\le 1$.

If $X$ is non minimal, for any element $\Theta\in \mathcal M_{mK_X}$, by Theorem 3 of \cite{Sik} (also see the proof of Proposition \ref{ES211}), there exists another tamed almost complex structure $J'$ for which $\Theta$ is $J'$-holomorphic and $J'$ is integrable on a neighborhood of $|\Theta|$. By Corollary 2.10 of \cite{p=h}, we always have a smooth $J'$-holomorphic $-1$ rational curve. Since $K_X\cdot E=-1<0$ for any $-1$ classes, we know this rational curve is a component of $\Theta$. We blow it down, and change the configuration of $\Theta$ correspondingly. Keep doing so until we get a minimal manifold $X'$, {\it i.e.} a manifold containing no smooth spheres of self-intersection $-1$. Hence, $\Theta$ changes to a pseudoholomorphic subvariety $\Theta_0$ in the nef class $mK_{\min}$. In particular, we know each irreducible component of $\Theta_0$ is a torus with square zero or a $-2$ sphere.  This implies each component in original $\Theta$ has non-positive self-intersection. Then we choose a point on each irreducible component $C_i$ of $\Theta$ with multiplicity $m_i$, we know there is a unique $J$-holomorphic subvariety in $\mathcal M_{mK_X}$ passing through them. Since there are at most $mC_{K_X}$ irreducible components for any elements in $\mathcal M_{mK_X}$, we know $X^{mC_{K}}$ contains a parameter space of $\mathcal M_{mK_X}$. This implies $\kappa^J_W(X)\le 1$.

Let us look at the case of $\kappa^s(X)=2$. We also first assume it is minimal. Again, $K_X$ is $J$-nef by Lemma 2.1 of \cite{p=h}.
Then we can choose $mK_X\cdot [C_i]+1$ points with multiplicities $m_i$ on the smooth part of each irreducible components of an element in $\mathcal M_{mK_X}$, we know there is a unique one passing through these points by Lemma \ref{pointscurve}. These points vary in a family of dimension at most $4\sum m_i(mK_X\cdot [C_i]+1)\le 4m^2K_X^2+4m\lceil\frac{[\omega]\cdot K_X}{\min_{C} [\omega]\cdot [C]}\rceil$.  In other words, $\dim \mathcal M_{mK_X}\le 4m^2K_X^2+4m\lceil\frac{[\omega]\cdot K_X}{\min_{C} [\omega]\cdot [C]}\rceil$. Hence $\kappa^J_W(X)\le 2$.

Now we assume $X$ is not minimal. By Lemma 2.1 of \cite{p=h}, if $K_X\cdot [C]<0$ for an irreducible curve $C$, then $C$ is a $-1$ rational curve. Hence, for any $\Theta\in \mathcal M_{mK_X}$, we choose $mK_X\cdot [C_i]+1$ points with multiplicity $m_i$ on the smooth part of each irreducible component $C_i$ which is not a $-1$ rational curve. By Lemma \ref{pointscurve}, we know there is a unique such element in $\mathcal M_{mK_X}$ passing through these points with multiplicities. Now these points vary in a family of dimension at most $4\sum m_i(mK_X\cdot [C_i]+1)$. Notice for any $\Theta\in \mathcal M_{mK_X}$, the sum of the multiplicities $m_i$ for components of $\Theta$
is bounded from above by $mC_{K_X}$. Hence, we have
\begin{align*} \dim \mathcal M_{mK_X} &\le 4\sum_{C_i\notin \mathcal E_{K_J}} m_i(mK_X\cdot [C_i]+1)\\
&=4m^2K_X^2+\sum_{C_i\in \mathcal E_{K_J}}4m_i(m+1)+\sum_{C_i\notin \mathcal E_{K_J}} 4m_i\\
&\le 4m^2K_{\min}^2-4km^2+4m^2C_{K_X}+4mC_{K_X}\\
&=(4K_{\min}^2-4k+4C_{K_X})m^2+4mC_{K_X}
\end{align*}
where $k$ is the number of $-1$ classes in $X$. This implies $\kappa_W^J(X)\le 2$.
\end{proof}

\begin{proof} [Proof of Theorem \ref{JWs}]
It follows from Lemma \ref{0tosec} and Theorem \ref{W<s}.
\end{proof}

The estimates in the above proof might remind the reader about estimates in probabilistic combinatorics. The coefficients of the polynomials in our estimates are not optimal.

\begin{remk}\label{W=s}

In the proof of Theorem \ref{W<s}, we have shown that $\kappa_W^J(X)= \kappa^s(X)$ when $\kappa^s(X)\le 0$.

Moreover, when $(X, J)$ is a tamed almost complex $4$-manifold with $b^+=1$ and $\kappa^s(X)=2$, the equality also holds. The result follows from applying Proposition \ref{bigKod2} to $e=K_{X_{min}}$.

\end{remk}

\begin{cor}\label{W<J}
Let $J$ be a tamed almost complex structure on a complex surface $(X, J_X)$. Then $$\kappa^J(X)\le \kappa^{J_X}(X).$$
\end{cor}
\begin{proof}
By \cite{DZ}, $\kappa^{J_X}(X)=\kappa^s(X)$ for any $4$-manifold admitting both complex and symplectic structures. Then the result follows from Theorem \ref{JWs}.
\end{proof}

\section{Vanishing of plurigenera}\label{van}
A negative line bundle on a compact complex manifold is a holomorphic line bundle whose Chern
class is a negative class. The Iitaka dimension of a negative line bundle is always negative infinity. For a complex line bundle $L$ over an almost complex manifold, we would prove similar vanishing results if it satisfies certain negative condition. First, for tamed almost complex structures, we have the following.

\begin{prop}\label{neginfty}
Let  $L$ be a complex line bundle over a compact almost complex $4$-manifold $(X, J)$. If there is a $J$-tamed symplectic form $\omega$ such that $[\omega]\cdot c_1(L)<0$, then for any bundle almost complex structure $\mathcal J$ on $L$, $P_{L, \mathcal J}=0$ and $\kappa^J(X, (L, \mathcal J))=-\infty$.
\end{prop}
\begin{proof}
If $P_{L, \mathcal J}>0$, then there is a nonzero section in $H^0(X,(L,\mathcal J))$. By Corollary 4.2 in \cite{CZ} and Corollary 1.3 in \cite{Z2}, its zero locus is a $J$-holomorphic $1$-subvariety in class $c_1(L)$. However, $c_1(L)\cdot [\omega]<0$ which contradicts to the tameness of $J$.

Similarly $P_{L^{\otimes m}, \mathcal J}=0$ for $m>1$ and $\kappa^J(M, (L, \mathcal J))=-\infty$.
\end{proof}

In particular, if  there is a $J$-tamed symplectic form $\omega$ on an almost complex $4$-manifold $(X,J)$ such that $[\omega]=-c_1(L)$, then for any bundle almost complex structure $\mathcal J$ on $L$, $P_{L, \mathcal J}=0$ and $\kappa^J(M, (L, \mathcal J))=-\infty$. As a specific example, any tamed $J$ on $\mathbb CP^2$ has the plurigenera $P_m=0, m>0,$ and the Kodaira dimension $\kappa^J=-\infty$ (this also follows from Corollary \ref{W<J}).

 To generalize the above to higher dimensions, one effective way is to prove a result similar to Corollary 1.3 in \cite{Z2} for almost complex submanifolds with any codimensions. Another way is to apply the analytic method. We prove the following result which only uses Theorem 3.8 in \cite{Z2} and a Bochner type formula on almost complex manifolds.

\begin{thm} \label{c1<0van}

Let  $L$ be a complex line bundle over a compact almost complex $2n$-manifold $(X, J)$. If there is a $J$-compatible symplectic form $\omega$ such that $[\omega]^{n-1}\cdot c_1(L)<0$, then for any bundle almost complex structure $\mathcal J$ on $L$, $P_{L, \mathcal J}=0$ and $\kappa^J(X, (L, \mathcal J))=\kappa_J(X, (L, \mathcal J))=-\infty$.
\end{thm}

\begin{proof}

We prove it by contradiction. Assume that $s\in H^0(X,(L,\mathcal J))$ is a nonzero section. Denote $V=\{x\in X| s(x)=0\}$. By Theorem 3.8 in \cite{Z2}, $V$ is a closed set with Hausdorff dimension $2n-2$ (if nonempty). In particular, $V$ has measure zero and the complement $U=X\setminus V$ is a densely open set.

Assume that $h$ is an arbitrary Hermitian metric on $L$ with $|s|_h^2=h(s,s)\geq 0$. Let $g$ be the almost K\"ahler metric on $X$ determined by $\omega$ and $J$, and denote $\nabla$ the unique Hermitian connection on $L$ such that $\nabla^{0,1}=\bar{\p}_L$ (see Lemma 3.3 in \cite{CZ}). On $U$, we have $\nabla s=\nabla^{1,0}s=\xi\otimes s$, for some $\xi\in \Lambda^{1,0}(U)$. From $$
d|s|_h^2=\nabla |s|_h^2=h(\nabla s, s)+h(s, \nabla s)=\xi |s|_h^2+\bar{\xi}|s|_h^2,
$$
we get that $\xi=\dfrac{\p |s|_h^2}{|s|_h^2}=\p \log |s|_h^2$.

Denote the curvature form of $h$ on $L$ by $\Theta$. Then $\Theta|_U=d\xi=d\p \log |s|_h^2$. By the Chern-Weil theory, we have $c_1(L)=[\dfrac{\sqrt{-1}}{2\pi}\Theta]$.
Also, on $U$, we have \begin{align*}
\sqrt{-1}\tr_\omega(\bar{\p}\p\log |s|_h^2) &=\sqrt{-1}(\dfrac{tr_\omega (\bar{\p}\p |s|_h^2)|s|_h^2 - \tr_\omega (\bar{\p}|s|_h^2\wedge \p |s|_h^2)}{|s|_h^4})\\
&=\sqrt{-1}\dfrac{\tr_\omega (\bar{\p}\p |s|_h^2)}{|s|_h^2}+|\xi|^2,
\end{align*} where $\tr_\omega\cdot$ denotes the contraction operator on a $(1,1)$ form by $\omega^{-1}$.
So we get the identity \begin{align} -\sqrt{-1}\tr_\omega (\bar{\p}\p |s|_h^2)=|\nabla s|_h^2-\sqrt{-1}\tr_\omega(\bar{\p}\p\log |s|_h^2)|s|_h^2. \label{111} \end{align}
For any function $f$, since $d\p f=-d\bar{\p}f$, we have $\p\bar{\p}f=-\bar{\p}\p f$. Denote $\Delta_\omega=-\sqrt{-1}\tr_\omega \bar{\p}\p=\sqrt{-1}\tr_\omega \p\bar{\p}$ acting on smooth functions of $X$. It is an elliptic operator which differs from the Levi-Civita Laplacian by first order terms (see Lemma 3.2 in \cite{T}). So the maximum principle also holds in this case. Extend $\tr_\omega$ to any 2-form $\phi$ by defining $\tr_\omega \phi=\tr_\omega \phi^{1,1}$, where $\phi^{1,1}$ is the $(1,1)$ component of $\phi$. From $(\ref{111})$ and the denseness of $U$, we obtain the Bochner type formula on $X$ (in the integrable case, see \cite{Y1} for example): \begin{align} \Delta_\omega |s|_h^2=|\nabla s|_h^2-\sqrt{-1}\tr_\omega\Theta|s|_h^2 \label{Bo} \end{align}
Now by assumption, we have
\begin{align*}
0>[\omega]^{n-1}\cdot c_1(L) &=\int_X \dfrac{\sqrt{-1}}{2\pi}\Theta\wedge \omega^{n-1}=\int_X \dfrac{\sqrt{-1}}{2\pi}\Theta^{1,1}\wedge \omega^{n-1}\\
&=\int_X \dfrac{\sqrt{-1}\tr_{\omega}\Theta^{1,1}}{2n\pi} \omega^{n}=\int_X \dfrac{\sqrt{-1}\tr_{\omega}\Theta}{2n\pi} \omega^{n}
\end{align*}
This gives that the average of $\sqrt{-1}\tr_\omega\Theta$ is negative. Next, we do a conformal change of the metric $h$ to make it to be a negative constant (see similar arguments in \cite{Gau, Y1}). To obtain this, for any smooth real function $\varphi$ on $X$, $h'=e^{\varphi}h$ is also a Hermitian metric on $L$. The Hermitian connection of $h'$ has connection one form $\xi'=\xi+\p \varphi$ and curvature form $\Theta'=\Theta+d\p \varphi$. Then on $X$, $\sqrt{-1}\tr_\omega\Theta'=\sqrt{-1}\tr_\omega\Theta-\Delta_\omega \varphi$. Let $k=\int_X \dfrac{\sqrt{-1}\tr_{\omega}\Theta}{2n\pi} \omega^{n}<0$. Since $\omega$ is closed (a Gauduchon metric is sufficient), $\int_X \Delta_\omega f\omega^{n}=0$ for any smooth function $f$. By the Hodge decomposition of smooth functions with respect to $\Delta_\omega$, the Poisson equation $\Delta_\omega \varphi=\sqrt{-1}\tr_\omega\Theta-\dfrac{2nk\pi}{\int_X \omega^n}$ has a smooth solution. Choosing this solution $\varphi$, we get that $$\sqrt{-1}\tr_\omega\Theta'=\dfrac{2nk\pi}{\int_X \omega^n}=k_1,$$ where $k_1$ is a negative constant.

Return to the Bochner formula for the Hermitian metric $h'$, we have $$\Delta_\omega |s|_{h'}^2=|\nabla s|_{h'}^2-k_1 |s|_{h'}^2.$$
Since $X$ is compact, there is a maximum point $p$ of $|s|_{h'}^2$ with $|s|_{h'}^2(p)>0$. By the maximum principle, $\Delta_\omega |s|_{h'}^2(p)\leq 0$, but $(|\nabla s|_{h'}^2-k_1 |s|_{h'}^2)(p)>0$, a contradiction. Therefore, $H^0(X, (L,\mathcal J))=0$.

As $c_1(L^{\otimes m})\cdot [\omega]^{n-1}=mc_1(L)\cdot [\omega]^{n-1}<0$, we similarly get that $H^0(X, (L^{\otimes m},\mathcal J))=0$ and $\kappa^J(X, (L, \mathcal J))=-\infty$.
\end{proof}

From the proof above, we also prove the following vanishing of plurigenera, which is a generalization of the known results in the complex setting (in the K\"ahler case by Yau \cite{Yau} and in more general Hermitian case by Yang \cite{Y}).
\begin{thm} \label{metric}
Let $(X,J)$ be a compact almost complex manifolds. If one of the following is satisfied:
\begin{enumerate}
\item $X$ holds a Hermitian metric with positive Chern scalar curvature everywhere,
\item $X$ holds a Gauduchon metric with positive total Chern scalar curvature,
\end{enumerate}
then $\kappa^J(X)=\kappa_J(X)=-\infty$.

\end{thm}

\begin{proof}
Let $\omega$ be an almost Hermitian metric on $X$. It naturally induces a Hermitian metric on the canonical bundle $\mathcal K$ as discussed in Section 3. The unique Hermitian connection $\nabla$ on $\mathcal K$ is induced by the Chern connection (see {\it e.g.} Remark 3.4 in \cite{CZ}). Denote $\Theta$ the curvature 2-form. It is well known \cite{T} that $-\sqrt{-1}\tr_\omega \Theta=s^C$ where $s^C$ is the Chern scalar curvature of the Chern connection. Let $\sigma\in H^0(X,\mathcal K_X^{\otimes m})$. The same calculation as above gives that
$$\Delta_\omega |\sigma|^2=|\nabla \sigma|^2+ms^C|\sigma|^2.$$ If $s^C>0$, using the maximum principle again, we derive that $\sigma=0$.

Next, assume that the total Chern scalar curvature $k=\int_X s^C \omega^n>0$ with $\omega$ being a Gauduchon metric, i.e. $\p\bar{\p} \omega^{n-1}=0$. Denote $\Delta_\omega=\sqrt{-1}\tr_\omega \p\bar{\p}$ the generalized Laplacian. For any smooth function $f$, as $$\int_X \Delta_\omega f \omega^n=\int_X \sqrt{-1}n\p\bar{\p}f\wedge \omega^{n-1}=\int_X \sqrt{-1}nf\wedge \p\bar{\p}\omega^{n-1}=0,$$
where we use the Gauduchon condition in the last step, the Poisson equation $\Delta_\omega \varphi=-s^C+\frac{k}{\int_X \omega^n}$ has a smooth solution $\varphi$. Define a Hermitian metric on $\mathcal K^{\otimes m}$ by $h'=e^{\varphi}h_\omega$, where $h_\omega$ denotes the Hermitian metric induced by $\omega$. Then the curvature form $\Theta'$ of $h'$ has $\sqrt{-1}\tr_\omega \Theta'=\sqrt{-1}\tr_\omega \Theta-\Delta_\omega \varphi=-\frac{k}{\int_X \omega^n}<0$. Let $\sigma\in H^0(X,\mathcal K^{\otimes m})$. The Bochner formula with respect to $h'$ then reads:
$$\Delta_\omega |\sigma|_{h'}^2=|\nabla \sigma|_{h'}^2+\frac{k}{\int_X \omega^n}|\sigma|_{h'}^2.$$ Therefore, the maximum principle gives that $\sigma=0$. We finish the proof.
\end{proof}

\section{Further discussion and open problems}\label{pro}
In this last section, we would discuss some speculations and open problems generated from this series of papers. For more discussions and in particular the relation with the intersection theory of almost complex manifolds and other aspects of almost complex geometry, we recommend \cite{Ziccm}.
\subsection{Hodge theory}\label{secHodge}
It would be very interesting to compare the cohomology group $H^0(X, \mathcal K_X)\cong H^{2,0}_{\bar\partial}(X)$ with the $J$-anti-invariant cohomology $H_J^-(X)$ defined in \cite{LZcag} when $\dim X=4$. Both are birational invariants \cite{CZ, BZ}. The zero loci of elements in both groups support $J$-holomorphic subvarieties in the canonical class. For $H_J^-(X)$, both statements are shown in \cite{BZ}. The non-integrable almost complex structures of Section 6.3 of \cite{CZ} on $T^2\times S$ have $\dim H^0(X, \mathcal K_X)\ge g-1$. On the other hand, Conjecture 2.5 in \cite{DLZ} says the dimension of the $J$-anti-invariant cohomology $H_J^-(X)$ is greater than 2 only if $J$ is integrable. It would be interesting to calculate $H_J^-(X)$ for these examples explicitly.

In higher dimensions, we could similarly define a cohomology group $MH_J^-(X)$ ($M$ stands for middle). We look at the bundle of real parts of $(n,0)$ forms $\Lambda_{\R}^{n,0}$ where $n=\dim_{\C} X$, whose space of sections is denoted by $\Omega_{\R}^{n,0}$. The almost complex structure $J$ on $X$ induces a complex line bundle structure on $\Lambda_{\R}^{n,0}$. It can be described concretely by its action on a section $\beta $ as $J\beta(v_1, v_2, \cdots, v_n):=-\beta(Jv_1, v_2, \cdots, v_n)$. We define $MH_J^-(X)$ to be the subgroup of $H^{n}(X, \R)$ whose element has a representative in $\Omega_{\R}^{n,0}$.
 For any Riemannian metric $g$ compatible with $J$, $\Lambda_{\R}^{n,0}$ is a subbundle of  $\Lambda_g^+$. Hence, a closed form in $\Omega_{\R}^{n,0}$ is a self-dual harmonic form. As discussed in \cite{BZ}, any representative of an element in $MH_J^-(X)$ is a $J$-holomorphic subvariety of complex dimension $n-1$ in the canonical class up to Question 3.9 in \cite{Z2}. It would be interesting to compare $H^0(X, \mathcal K_X)\cong H^{n,0}_{\bar\partial}(X)$ with  $MH_J^-(X)$.

It would be great to understand more about the groups $\mathcal{H}_{\bar{\p}_{E}}^{(p,q)}(X, E)$. When $E$ is a holomorphic bundle over a complex manifold $X$, they are isomorphic to the Dolbeault cohomology groups of holomorphic bundles by Hodge theory. Especially, when $E$ is the trivial line bundle, these are just the classical Dolbeault cohomology groups. When $q=0$ or $\dim_{\C}X$, the groups depend only on $J$ and $\bar{\p}_E$ \cite{CZ}, in particular, it is independent of the defining Hermitian metric. Whether its dimension is also an invariant of the almost complex structure $J$ when $E$ is a trivial bundle and $0<q<\dim_{\C}X$ is a famous question of Kodaira-Spencer in Hirzebruch's 1954 problem list \cite{Hir}, which is answered negatively recently by Tom Holt and the second author \cite{HZ}. The method works very effectively to compute $\mathcal{H}_{\bar{\p}}^{(p,q)}(X)$ explicitly for nilmanifolds. It would be great to use their techniques to compute general $\mathcal{H}_{\bar{\p}_{E}}^{(p,q)}(X, E)$.

Many aspects of Dolbeault cohomology groups for complex manifolds extend to our setting. For example, it is clear from the definition that $\mathcal{H}_{\bar{\p}_E}^{(n,q)}(X, E)\cong \mathcal{H}_{\bar{\p}_E}^{(0,q)}(X, \mathcal K_X\otimes E)$. When $q=0$ and $E$ is the trivial line bundle, it is equal to $H^0(X, \mathcal K_X)$. We also have Serre duality for the groups $\mathcal{H}_{\bar{\p}_{E}}^{(p,q)}(X, E)$, namely, $\mathcal{H}_{\bar{\p}_{E}}^{(p,q)}(X, E)\cong (\mathcal{H}_{\bar{\p}_{E^*}}^{(n-p,n-q)}(X, E^*))^*$ (see Proposition 3.7 in \cite{CZ}).

Another interesting direction is to show the dimension of $\mathcal{H}_{\bar{\p}}^{(p,0)}(X):=\mathcal{H}_{\bar{\p}_{\mathcal O}}^{(p,0)}(X, \mathcal O)$ is a birational invariant, namely it is invariant under degree $1$ pseudoholomorphic maps. It is shown in Theorem 5.5 in \cite{CZ} for $\dim X=4$.

\subsection{Comparison of Iitaka dimensions}
We have defined three versions of Iitaka dimension for almost complex manifolds. The versions $\kappa^J$ and $\kappa_J$ are defined for arbitrary almost complex structure $J$. The last version $\kappa^J_W$ is defined for tamed almost complex structures on $4$-manifolds. One is able to formulate a similar definition of $\kappa^J_W$ in higher dimensions. Generalization of positivity of intersection in \cite{Z2} to singular subvarieties should lead to similar results as in dimension $4$.

Question \ref{k=k} asks whether $\kappa_J(X, (L, \mathcal J))=\kappa^J(X, (L, \mathcal J))$. On the other hand, $\kappa^J_W(X, c_1(L))$ is in general larger. By definition, $\kappa_J$ takes value from $-\infty$ and integers $0, 1, \cdots, n$. However, we do not know whether $\kappa^J(X, (L, \mathcal J))\le \dim_{\C}X$ and whether it always takes integer values, although we expect these are true because of Question \ref{k=k}. We also would like to show $\kappa^J_W(X, e)\le 2$ in general. By Theorem \ref{W<s}, It is true for the corresponding Kodaira dimension, {\it i.e.} when $e=K_X$. In particular, it implies $\kappa^J(X)\le 2$ when $(X, J)$ is a tamed almost complex manifold.

The second version of our Kodaira dimensions $\kappa_J$ is the simplest invariant provided by the image of $\Phi_{L, \mathcal J}$. We would like to know more about this image. It has been shown that $\Phi_{L, \mathcal J}(X\setminus B)$ is an open analytic subvariety. It is believed to be true that $\overline{\Phi_{L, \mathcal J}(X\setminus B)}$ is a projective subvariety, as shown in Corollary \ref{proj} when $\dim_{\R} X=4$.

The study of this problem would also tell us what a rational map in almost complex setting should be like, as we have shown in Proposition \ref{pcm} for dimension $4$. In general, the study of pseudoholomorphic rational maps, {\it i.e.} pseudoholomorphic maps defined on the complement of a pseudoholomorphic subvariety in an almost complex manifold, is very important. The fundamental question is whether the closure of its image and the complement of the image in it are pseudoholomorphic subvarieties in the target almost complex manifold.

In addition, we would like to know whether the transcendental degree of the quotient field of the section ring is equal to $\kappa_J$. In particular, we expect (the closure of) the image of the pluricanonical map $\Phi_m$ would be stable for large $m$, and equal to $Proj (R(X, K_X))$.

\subsection{Extension of pseudoholomorphic sections}
Hartogs type extension is very important in studying our Kodaira dimensions. For example, we would like to know

\begin{que}\label{Hartogs}
If $(L, \mathcal J)$ is a pseudoholomorphic bundle over $(X, J)$ and $B\subset X$ is a pseudoholomprhic subvariety of complex codimension at least two, does any pseudoholomorphic section of $(L, \mathcal J)$ over $X\setminus B$ extend to a pseudoholomorphic section over $X$?
\end{que}

Theorem 5.2 in \cite{CZ} establishes this Hartogs extension when $\dim X=4$. In general, Question \ref{Hartogs} would help to prove the invariance of $\kappa^J$ and $\kappa_J$ under degree $1$ pseudoholomorphic maps, {\it i.e.} the birational invariance of $\kappa^J$ and $\kappa_J$, in all dimensions. It would be helpful in obtaining other results, {\it e.g.} Question \ref{k=k}. It would also help to answer the following question.

\begin{que}
Let $f:X\dashrightarrow M$ be a generically surjective pseudoholomorphic rational map such that $\dim X=\dim M$. Do we have $\kappa^J(X)\ge \kappa^J(M)$.
\end{que}
This answer is even not known when $f$ is defined on whole $X$. Notice we do not have the generalization of Iitaka conjecture. For any pseudoholomorphic surface bundle over surface $f: X\rightarrow M$ whose total space $X$ is not complex and both base and fiber have genus greater than one, we have $\kappa^J(F)+\kappa^J(M)=2>\kappa^J(X)$. Hence,  the additivity of Kodaira dimension, as explored in \cite{LZadd, Z1} for other kinds of Kodaira dimensions, would not hold for our almost complex Kodaira dimensions.

Another important type of extension result in complex geometry is the extendability of analytic subvarieties, for example Theorem \ref{Shiff}. We expect it also holds for pseudoholomorphic subvarieties. More generally, we would like to study the extendability of pseudoholomorphic maps between almost complex manifolds, generalizing the classical Picard theorem.

\subsection{Examples}
We would like to calculate more examples. We are particularly interested in finding more or even obtaining structural results for non-integrable almost complex structures with large Kodaira (or Iitaka) dimensions. For instance, do we have a version of Iitaka fibration? Do we know $\kappa^J=n-1$ only when the manifold is a pseudoholomorphic elliptic fibration over a complex manifolds? In dimension $4$, we have the second part of Theorem \ref{ellintro}.

In Section 6.3 of \cite{CZ}, we have examples of non-integrable almost complex structures with large Kodaira dimension. These examples are deformations of complex structures. It would be very interesting to know some non-integrable almost complex structures on symplectic non-K\"ahler manifolds with large Kodaira dimension. Some special pseudoholomorphic elliptic fibrations might work.

The values of Kodaira dimensions and dimensions of holomorphic $p$-forms for some special ambient space are also very interesting. In particular, we are interested in almost complex structures on $\mathbb{C}P^n$ and $S^6$. On $S^6$, there is a standard non-integrable almost complex structure denoted by $\mathsf J$. In Section 7 of \cite{CZ}, we compute Kodaira dimensions of $\mathsf J$ (and then $-\mathsf J$) as well as the value of the Hodge numbers $h^{1,0}, h^{2,0}$. Similar results are proven to hold on strictly nearly K\"{a}hler 6-manifolds (see \cite{CW}). It is known that there are two connected components of all the almost complex structures on $S^6$ \cite{Br3}, containing $\mathsf J$ and $-\mathsf J$ respectively. We would like to know whether our calculation would help to compute Kodaira dimension for other almost complex structures on $S^6$
and almost complex structures on other homogeneous spaces.

We would like to calculate examples of the Iitaka dimension for a general complex line bundle.  For example, we are interested in the following almost complex structures on homogeneous spaces.

Assume that $G$ is a compact connected semisimple Lie group and $T\subset G$ a maximal torus. Let $\mathfrak g$ and $\mathfrak t$ be Lie algebras of $G$ and $T$ respectively.
The set of all roots $R$ is by definition constituted of all elements $\alpha\in i\mathfrak t\subset \mathfrak t\otimes \C$ such that there exists a nonzero $X\in \mathfrak g\otimes \C$ with $[H, X]=\langle \alpha, H\rangle X$ for all $H\in \mathfrak t$, or equivalently, the characters of the corresponding irreducible representation of $T$, if we identify the character group $Hom(T, \C^*)$ with $(\C^*)^{\dim T}\subset \mathfrak t\otimes \C$. The eigenspaces are denoted $\mathfrak g_{\alpha}$.

The tangent space of $G/T$ at the coset $eT$ decomposes, in terms of Lie algebra, as $$(\mathfrak g/\mathfrak t)\otimes \C=\oplus_{\alpha\in R}\mathfrak g_{\alpha}.$$
By left action, such a decomposition induces a decomposition of $T^*(G/T)\otimes \C$. By choosing a regular element $H_0\in i\mathfrak t$, we have the decomposition of $R=R^+\cup R^-$ to positive and negative roots. Notice $R^+=-R^-$ as $\alpha\in R$ if and only if $-\alpha \in R$. We have
\begin{equation}\label{NNLie}
[\mathfrak g_{\alpha}, \mathfrak g_{\beta}]\subset \mathfrak g_{\alpha+\beta}
\end{equation}
 for any positive roots $\alpha, \beta$.

In fact, the decomposition $R=R^+\cup R^-$ is a decomposition $T(G/T)\otimes \C=T_{1, 0}(G/T)\oplus T_{0, 1}(G/T)$, thus gives rise an almost complex structure on $G/T$. The relation \eqref{NNLie} is the integrability condition. This is the well known complex structure on $G/T$.

We can modify this construction a bit such that $R$ is decomposed as a disjoint union of two sets $R_1$ and $R_2$, with the relation $R_1=-R_2$. But we do not require the condition  \eqref{NNLie}. This gives different choices of non-integrable almost complex structures on $G/T$. What is the Iitaka dimension of the anti-canonical bundle, or some other complex line bundle?

\subsection{Symplectic and almost contact invariants}
We can define symplectic invariants by $$\tilde{\kappa}^s(X, \omega)=\sup_{J \tiny{\hbox{ tamed by }} \omega} \kappa,$$ where $\kappa$ could be either $\kappa^J$, $\kappa_J$ or $\kappa^J_W$. We would like to explore more about it. In particular, whether these symplectic invariants are diffeomorphism invariants in dimension $4$? The version for $\kappa^J_W$ is better understood and studied in Section \ref{Weil}. One can also define symplectic invariants in a similar manner for $h^{p, 0}$.

We are also able to define invariants for almost contact structures in the following way. For any almost contact manifold $(M^{2n-1}, J)$, the product $M^{2n-1}\times S^1$ admits a natural almost complex structure. Then our almost complex invariants $\kappa^J$, $\kappa_J$, $\kappa^J_W$ and $h^{p,0}$ on $M\times S^1$ would induce almost contact invariants on $(M, J)$. It is particularly interesting to study them for all $3$-manifolds.


\newpage
\appendix
\section{Uniqueness of exceptional curves for irrational symplectic $4$-manifolds}
In this appendix, we would give an affirmative answer to the following question, which is Question 4.18 in \cite{p=h}.

\begin{que}\label{Es2}
Suppose $M$ is symplectic $4$-manifold, which is not diffeomorphic to $\mathbb CP^2\#k\overline{\mathbb CP^2}$. Let $E\in \mathcal E_{K_J}$. Is it true that for any subvariety $\Theta=\{(C_i, m_i)\}$ in class $E$, {\it i.e.} $E=\sum m_i[C_i]$, each $C_i$ is a smooth rational curve?
\end{que}

Here, a class $E\in \mathcal E_{K_J}$ (called an exceptional curve class) is a $K_J$-spherical class such that $E^2=K_J\cdot E=-1$, where $K_J$ is the canonical class of $J$ and a {\it $K_J$-spherical class} is a class $e$ which could be represented by a smoothly embedded sphere and $g_J(e):=\frac{1}{2}(e\cdot e+K_J\cdot e)+1=0$. For a generic tamed $J$, any exceptional curve class is represented by a unique embedded $J$-holomorphic sphere with self-intersection $-1$.

The answer to Question \ref{Es2} was known to be true for irrational ruled surfaces \cite{s2mod}, and there are even many integrable complex structures on rational surfaces where the statement does not hold \cite{p=h, s2mod}.

Before answering Question \ref{Es2}, we first prove the following result, which is essentially Proposition 2.11 of \cite{ES}.

\begin{prop}\label{ES211}
Suppose that $(M, \omega)$ is a non rational or ruled symplectic $4$-manifold and $J$ is an $\omega$-tame almost complex structure. If $\Theta\in \mathcal M_E$ for a class $E\in \mathcal E_{K_J}$, then $\Theta$ is an exceptional curve of the first kind.
\end{prop}
Recall that a $J$-holomorphic subvariety $\Theta$ is called an exceptional curve of the first kind if there exists a neighborhood $(N, J')$ of $\Theta$ where $J'$ is an integrable complex structure with $J'|_{|\Theta|}=J_{|\Theta|}$, an open neighborhood $N'$ of $0\in \mathbb C^2$, and a holomorphic birational map $\pi: N\rightarrow N'$, such that $\Theta$ is a $J'$-holomorphic $1$-subvariety and $\pi^{-1}(0)=\Theta$.
\begin{proof}
We first show that there exists a tamed almost complex structure $J'$ for which $\Theta$ is $J'$-holomorphic and $J'$ is integrable on a neighborhood of $|\Theta|$. It essentially follows from Theorem 3 of \cite{Sik}, with a couple of remarks.  By Theorem 1 of \cite{Sik}, there exists a complex structure $J_1$ on a neighborhood $U_0$ of Sing$(|\Theta|)$ coinciding with $J$ on Sing$(|\Theta|)$ and such that $|\Theta|\cap U_0$ is a $J_1$-complex curve. The first remark is that we can choose $J_1$ and $U_0$ such that it is arbitrarily closed to $J$ in the $C^0$ norm. This follows because it is clear from original construction of Micallef-White \cite{MW} that the $C^1$-diffeomorphism of $U_0$ and a neighborhood of the origin in the complex plane has Jacobian the identity matrix at the singular points. Hence, by choosing $U_0$ small, $J_1$ is $C^0$ close to $J$ in $U_0$. We then extend $J_1$ to $M$ as an almost complex structure so that $J_1$ is close to $J$ on a neighborhood $U$ of $|\Theta|$ and the regular part $C_1$ of $|\Theta|$ is still $J_1$-complex. We can assume that $C_1\cap U_0$ is a finite union of punctured disks $D_k$.

As in \cite{Sik}, we can assume Sing$(C)$ meets every component of $C$. Then $C_1$ has no closed component, thus its normal bundle is trivial as a complex bundle. Thus we can find a tubular neighborhood $T_1$ and a diffeomorphism $\tau: (T_1, C_1)\rightarrow (C_1\times \C, C_1\times \{0\})$. We transport the structure $J_1$ to obtain $\tau_*(J_1)=J_2$ on $C_1\times \C$ such that $C_1\times \{0\}$ is $J_2$-complex.

Let $J_0=J|_{C_1}\times i$ on $C_1\times \C$. Let $A_k\subset D_k$ be annuli (with respect to $J_1$), and $A$ be the union of $A_k$. Denote $A'$ and $A_k'$ respectively for those $J_0$-holomorphic curves that are $(J_0,J_1)$-biholomorphic to $A$ and $A_k$ on $C_1$ and $D_k$ respectively. The map could be chosen arbitrarily close to the identity by our construction. We then apply the extension lemma used in \cite{Sik} to  $A, A'$ to get a diffeomorphism $\psi$ between neighborhoods $N$ and $N'$ of $A$ and $A'$ that extends the previous one between $A$ and $A'$, such that $\psi_*(J_0)=\tau_*(J_1)$. We impose that it extends to a self-diffeomorphism (still denote $\psi$) of $C_1\times \C$, such that it is arbitrarily close to identity at $C_1\times \{0\}$.

Since the restriction of $\tau_*^{-1}\circ \psi_*$ on $C_1$ is the identity, by choosing the tubular neighborhood $T_1$ smaller, we can get $\tau_*^{-1}\circ \psi_*(J_0)$ arbitrarily $C^0$ close to $J$. Hence the complex structure
$$\tilde J=\begin{cases}\begin{array}{cl}

\tau_*^{-1}\circ \psi_*(J_0), & \ \text{on} \ T_2\setminus U_0\\
J_1, &\ \text{on } U_0
\end{array}
\end{cases}$$
defined as in \cite{Sik} (extend to whole $M$ by a cutoff function) is arbitrarily $C^0$ close to the original tamed $J$. Hence, it is tamed.

We then prove the proposition by induction on the number of irreducible components of $\Theta$. Let $\Theta=\{(C_i, m_i), 1\le i\le n\}\in \mathcal M_E$, where each $C_i$ is an irreducible $J$-holomorphic curve. We have $$-1=K\cdot E=K\cdot \Theta=\sum_{i=1}^nm_iK\cdot [C_i].$$ Hence, there is some $i$ such that $K\cdot [C_i]<0$. By Lemma 2.1 of \cite{p=h}, we know $C_i$ is a smooth rational $-1$-curve. We write $\pi: M\rightarrow M'$ for the blowing down map with respect to the above chosen $J'$ that is integrable on a neiborhood of $|\Theta|$.

We notice that if $\pi_*(\Theta)$ is not a point, then $K_{M'}\cdot \pi_*E=-1$. It follows from the calculation of Lemma 2.13 in \cite{ES} or the combinatorial blowdown calculation in \cite{p=h} as the classes of in $\mathcal E_{K_J}$ are orthogonal when $M$ is non rational or ruled. Then by induction, $\Theta$ can be blown down to a point. Therefore $\Theta$ is an exceptional curve of the first kind.
\end{proof}

By the above understanding of Sikorav's result, the argument of \cite{ES} would lead to the following more complete version of their main result.
\begin{thm}\label{ES1}
Let $X$ be a symplectic $4$-manifold with $K_X=[\omega]$. If there is a symplectic embedding of rational homology ball $\iota: B_{p, q}\rightarrow X$, such that the rational blow-up of  $\iota(B_{p, q})$ is non-rational, then  $$l\le 4K_X^2+7, $$ where $l$ is the length of the continued fraction expansion of $\frac{p^2}{pq-1}$. When, $p=n, q=1$ (then $l=n-1$), we get the stronger inequality $l\le 2K_X^2+1$.
\end{thm}
In particular, it would lead to the same bound of the length of Wahl singularities ({\it i.e.} cyclic quotient singularities of type $\frac{1}{p^2}(pq-1, 1)$), for which an optimal bound is obtained in \cite{RU}.

\begin{proof}
As remarked in \cite{ES}, by  Proposition \ref{ES211} and the above understanding of Sikorav's result, their argument is able to obtain the above statement when the rational blow-up of $\iota(B_{p, q})$ is non-ruled, by appealing to Lemma 2.1 of \cite{p=h} instead of their Corollary 2.2.

We can show that the rational blow-up of $\iota(B_{p, q})$, denoted by $W$, cannot be irrationally ruled. If it were, we blow down the exceptional curves $k$ times to get a minimal model $S$. We have $K_S^2\le 0$ since $W$ is  irrationally ruled. We have the relation $K_S^2-k=K_W^2=K_X^2-l$. As in \cite{ES}, the number $l$ is also the number of (negative) spheres $\{C_1, \cdots, C_l\}$ in the chain of exceptional divisor of the singularity of type $\frac{1}{p^2}(pq-1, 1)$. We can choose a tamed almost complex structure $J$ such that each $C_i$ is $J$-holomorphic. By Theorem 1.2 of \cite{s2mod}, each negative rational curve is an irreducible component of the finitely many (disjoint) reducible subvarieties in the fiber class $T$.  We can further choose this $J$ such that it is integrable in a neighborhood of the union of these reducible subvarieties. When $W$ is symplectically non-minimal, or equivalently, there is at least one reducible subvariety in the class $T$, by Corollary 2.10 of \cite{p=h}, we know there is at least one smooth $J$-holomorphic $-1$ rational curve (which must be an irreducible component of a reducible subvariety in the class $T$). Blow it down and continue this process, we known each reducible subvariety in the class $T$ can be reduced to a smooth sphere of square $0$ by blowdown. Hence, in particular, we know $l\le k$. Since $K_X^2=[\omega]^2>0$, we have $K_S^2=K_X^2+(k-l)>0$, which is a contradiction.
\end{proof}

We will now give an affirmative answer to Question \ref{Es2} also for other symplectic $4$-manifolds. More precisely, we have the following.

\begin{thm}\label{1E}
Let $M$ be a symplectic $4$-manifold which is not diffeomorphic to $\mathbb CP^2\#k\overline{\mathbb CP^2}, k\ge 1$. Then for any tamed $J$, there is a unique $J$-holomorphic subvariety in any class $E \in \mathcal E_{K_J}$, whose irreducible components are smooth rational curves of negative self-intersection. Moreover, this $J$-holomorphic subvariety is an exceptional curve of the first kind when $M$ is not ruled.
\end{thm}
\begin{proof}
When $M$ is an irrational ruled surface $S^2\times \Sigma_g\#k\overline{\mathbb CP^2}$, $g\ge 1$,  it follows from Theorem 1.1 of \cite{s2mod}.

We now prove the theorem when $M$ is neither rational nor ruled. In this situation, there is at lease one $J$-holomorphic subvariety in class $E$ which is an exceptional curve of the first kind as explained in the following. For generic tamed almost complex structure, $E$ is represented by a smooth rational curve. By Gromov compactness, we have a $J$-holomorphic stable map representing $E$. By Proposition \ref{ES211}, this is an exceptional curve of first kind. We call it $\Theta_0$. In particular, there is a smooth rational $-1$-curve $E'$. For convenience, we also call its homology class $E'$. Since $M$ is not rational or ruled, we know $E\cdot E'=0$.

Then we prove the uniqueness. Suppose there is another $J$-holomorphic subvariety in class $E$, say $\Theta=\{(C_i, m_i)\}_{i=1}^n$. By Theorem 3 of \cite{Sik}, there exists another tamed almost complex structure $J'$ for which $\Theta_0$ and $\Theta$ are both $J'$-holomorphic and $J'$ is integrable on a neighborhood of $|\Theta_0|\cup |\Theta|$. We may therefore assume without loss of generality that $J=J'$ in the following.

There are two cases of blowdowns. If there is some $C_i$ with $E'=[C_i]$, then we simply blow down $E'$ by $\pi_1:M\rightarrow M'$ and $\pi_{1*}([\Theta])=\pi_{1*}([\Theta_0])$ become an exceptional curve class $\pi_{1*}E$ in $M'$.

If $E'$ is not an irreducible component of $\Theta$, we have $E'\cdot [C_i]\ge 0, \forall i$. Since $E\cdot E'=0$, we have $E'\cdot [C_i]=0, \forall i$.
We blow down the curve $E'$ by $\pi_2: M\rightarrow \tilde M$. Now $\pi_{j*}(\Theta_0)$ and $\pi_{j*}(\Theta)$ are still $\pi_{j*}J$-holomorphic subvarieties representing $\pi_{j*}(E)$ for $j=1, 2$.  By induction and Lemma 2.13 of \cite{ES}, we can continue blowing down $\Theta_0$ to a smooth point. Call the composition of the sequence of (first and second types of) blowdowns by $\pi: M \rightarrow M_0$. We have $\pi_*([\Theta])=\pi_*([\Theta_0])=0$. Since $J$ is tamed, we know $\pi(|\Theta|)=\{pts\}$. In other words, we have proved that the support $|\Theta|\subset |\Theta_0|$. In other words, $\Theta$ and $\Theta_0$ share the same irreducible components, but their multiplicities (nonnegative, but could be zero) might be different.

Now we show that these multiplicities are in fact the same. For doing so, we rewrite the homology classes of each $C_i$. Since $\Theta_0$ is an exceptional curve of first kind, each $C_i$ corresponds to a $-1$-rational curve in each blowdown. In other words, there are $n$ second cohomology classes $E_1, \cdots, E_n$ which can be represented by embedded symplectic spheres with $E_i^2=-1$, $i=1, \cdots, n$, and $E_i\cdot E_j=0$ for $i\ne j$, such that $[C_i]=E_i-\sum_{j> i} m_{ij}E_j$. Here $(m_{ij})_{i, j=1}^n$ is a strict upper triangular square matrix whose entries are $0$ or $1$.
 In particular, $E=E_1$ and $E'=E_n=[C_n]$. Thus $[C_i]$, $i=1, \cdots, n$,  are linearly independent. Hence there is a unique way to represent the class $E$ by the linear combination of $[C_i]$. This implies each irreducible component of $\Theta$ and $\Theta_0$ has the same multiplicity, {\it i.e.} $\Theta=\Theta_0$.
Hence, we finish the proof of uniqueness when $M$ is neither rational nor ruled, and thus all the statements in the theorem.
\end{proof}

The result of Theorem \ref{1E} is sharp. When $M=\mathbb CP^2\#k\overline{\mathbb CP^2}$ with $k\ge 1$,  a $J$-holomorphic subvariety in class $E$ need not to be an exceptional curve of the first kind.

 If $M=\mathbb CP^2\#\overline{\mathbb CP^2}$, we look at a Hirzebruch surface $\mathbb F_3$. It contains a $-3$-curve in class $2E_1-H$. Then along with a fiber, it constitutes a $J$-holomorphic subvariety in class $E_1=(2E_1-H)+(H-E_1)$ which is not an exceptional curve of the first kind.

For $k= 2$, we blow up at a point not on the $-3$-curve of $\mathbb F_3$. If we call the blow up divisor $E_2$, then the subvariety $\{(2E_1-H, 1), (H-E_1-E_2, 1), (E_2, 1)\}\in \mathcal M_{E_1}$ is not an exceptional curve of the first kind. For $k>2$, we can continue this process by blowing up points outside the support of this subvariety.

When $S^2\times \Sigma_g\#k\overline{\mathbb CP^2}$, $g\ge 1$ and $k>1$, there are also examples of integrable complex structure $J$ such that a subvariety in class $E\in \mathcal E_{K_J}$ is not  an exceptional curve of the first kind. We first blow up a fiber to get two exceptional rational curves in classes $E_1$ and $T-E_1$ respectively. Then we blow up the intersection of these two rational curves. The subvariety $\{(T-E_1-E_2, 1), (E_2, 1), (E_1-E_2, 1)\}\in \mathcal M_{T-E_2}$ is not an exceptional curve of the first kind. This phenomenon is the essential underlying reason for the totally different proofs of Question \ref{Es2} for the irrational ruled case in \cite{s2mod} and the non ruled case in this paper.

When $M=\mathbb CP^2\#k\overline{\mathbb CP^2}$ with $k\ge 8$, even the statement of Question \ref{Es2} is not true, see \cite{s2mod}.

The argument of Theorem \ref{1E} could be used to prove other uniqueness results of subvarieties. The following is used in our proof of Theorem \ref{pointscurve}. We will skip some steps appeared in the proof of Theorem \ref{1E}.

\begin{prop}\label{1pcE}
Suppose $M^4$ is not rational or ruled, and $J$ is a tamed almost complex structure. Let $\{E_1, \cdots, E_n\}$ be the set of exceptional curve classes. Then there is a unique $J$-holomorphic subvariety in any class $k_1E_1+\cdots+k_nE_n$ where $k_i$ are non-negative integers.
\end{prop}
\begin{proof}
First, there exists a $J$-holomorphic subvariety $\Theta_0$ in any class $k_1E_1+\cdots+k_nE_n$ provided by a combination of $J$-holomorphic subvarieties in classes $E_i$.

Suppose there is another, say $\Theta=\{(C_i, m_i)\}_{i=1}^n$. As in the proof of Theorem \ref{1E}, we may assume $J$ is integrable on a neighborhood of $|\Theta_0|\cup |\Theta|$. We keep blowing down $\Theta_0$ to a smooth point, and denote the composition of the sequence of blowdowns by $\pi: M\rightarrow M_0$. We have $\pi_*([\Theta])=\pi_*([\Theta_0])=0$. Since $J$ is tamed, we know $\pi(|\Theta|)=\{pts\}$. In other words, we have proved that $\Theta$ and $\Theta_0$ share the same irreducible components. We can apply the same argument to show the multiplicities  are the same. Hence, we have shown $\Theta=\Theta_0$.
\end{proof}

\end{document}